\documentclass[12pt,twoside]{preprint}
\usepackage[margin=1.2in]{geometry}
\usepackage{amssymb}
\usepackage{mathtools}
\usepackage{hyperref}
\usepackage{breakurl}
\usepackage{mhenvs}
\usepackage{mhequ}
\usepackage{mhsymb}
\usepackage{booktabs}
\usepackage{tikz}
\usepackage{mathrsfs}
\usepackage{times}
\usepackage{microtype}
\usepackage{comment}
\usepackage{wasysym}
\usepackage{centernot}
\usepackage{enumitem}
\usepackage{bm}
\usepackage{stackrel}
\usepackage{yfonts}
\usepackage{url}

\DeclareMathAlphabet{\mathbbm}{U}{bbm}{m}{n}

\DeclareFontFamily{U}{BOONDOX-calo}{\skewchar\font=45 }
\DeclareFontShape{U}{BOONDOX-calo}{m}{n}{
  <-> s*[1.05] BOONDOX-r-calo}{}
\DeclareFontShape{U}{BOONDOX-calo}{b}{n}{
  <-> s*[1.05] BOONDOX-b-calo}{}
\DeclareMathAlphabet{\mcb}{U}{BOONDOX-calo}{m}{n}
\SetMathAlphabet{\mcb}{bold}{U}{BOONDOX-calo}{b}{n}

\makeatletter
\pgfdeclareshape{crosscircle}
{
  \inheritsavedanchors[from=circle] 
  \inheritanchorborder[from=circle]
  \inheritanchor[from=circle]{north}
  \inheritanchor[from=circle]{north west}
  \inheritanchor[from=circle]{north east}
  \inheritanchor[from=circle]{center}
  \inheritanchor[from=circle]{west}
  \inheritanchor[from=circle]{east}
  \inheritanchor[from=circle]{mid}
  \inheritanchor[from=circle]{mid west}
  \inheritanchor[from=circle]{mid east}
  \inheritanchor[from=circle]{base}
  \inheritanchor[from=circle]{base west}
  \inheritanchor[from=circle]{base east}
  \inheritanchor[from=circle]{south}
  \inheritanchor[from=circle]{south west}
  \inheritanchor[from=circle]{south east}
  \inheritbackgroundpath[from=circle]
  \foregroundpath{
    \centerpoint%
    \pgf@xc=\pgf@x%
    \pgf@yc=\pgf@y%
    \pgfutil@tempdima=\radius%
    \pgfmathsetlength{\pgf@xb}{\pgfkeysvalueof{/pgf/outer xsep}}%
    \pgfmathsetlength{\pgf@yb}{\pgfkeysvalueof{/pgf/outer ysep}}%
    \ifdim\pgf@xb<\pgf@yb%
      \advance\pgfutil@tempdima by-\pgf@yb%
    \else%
      \advance\pgfutil@tempdima by-\pgf@xb%
    \fi%
    \pgfpathmoveto{\pgfpointadd{\pgfqpoint{\pgf@xc}{\pgf@yc}}{\pgfqpoint{-0.707107\pgfutil@tempdima}{-0.707107\pgfutil@tempdima}}}
    \pgfpathlineto{\pgfpointadd{\pgfqpoint{\pgf@xc}{\pgf@yc}}{\pgfqpoint{0.707107\pgfutil@tempdima}{0.707107\pgfutil@tempdima}}}
    \pgfpathmoveto{\pgfpointadd{\pgfqpoint{\pgf@xc}{\pgf@yc}}{\pgfqpoint{-0.707107\pgfutil@tempdima}{0.707107\pgfutil@tempdima}}}
    \pgfpathlineto{\pgfpointadd{\pgfqpoint{\pgf@xc}{\pgf@yc}}{\pgfqpoint{0.707107\pgfutil@tempdima}{-0.707107\pgfutil@tempdima}}}
  }
}
\makeatother

\colorlet{symbols}{black}      
\colorlet{testcolor}{green!60!black}

\def\BPHZ{\textnormal{\tiny \textsc{bphz}}}

\def\symbol#1{\textcolor{symbols}{#1}}

\def\1{\mathbf{\symbol{1}}}

\usetikzlibrary{shapes.misc}
\usetikzlibrary{shapes.symbols}
\usetikzlibrary{shapes.geometric}
\usetikzlibrary{snakes}
\usetikzlibrary{decorations}
\usetikzlibrary{decorations.markings}

\def\mfL{\mathfrak{L}}

\def\drawx{\draw[-,solid] (-3pt,-3pt) -- (3pt,3pt);\draw[-,solid] (-3pt,3pt) -- (3pt,-3pt);}
\tikzset{
	root/.style={circle,fill=testcolor,inner sep=0pt, minimum size=2mm},
	dot/.style={circle,fill=symbols,draw=symbols,inner sep=0pt,minimum size=0.4pt},
	var/.style={circle,fill=black!10,draw=black,inner sep=0pt, minimum size=2mm},
	dotred/.style={circle,fill=black!50,inner sep=0pt, minimum size=2mm},
	generic/.style={semithick,shorten >=1pt,shorten <=1pt},
	dist/.style={ultra thick,draw=testcolor,shorten >=1pt,shorten <=1pt},
	testfcn/.style={ultra thick,testcolor,shorten >=1pt,shorten <=1pt,<-},
	testfcnx/.style={ultra thick,testcolor,shorten >=1pt,shorten <=1pt,<-,
		postaction={decorate,decoration={markings,mark=at position 0.6 with {\drawx}}}},
	kprime/.style={semithick,shorten >=1pt,shorten <=1pt,densely dashed,->},
	kprimex/.style={semithick,shorten >=1pt,shorten <=1pt,densely dashed,->,
		postaction={decorate,decoration={markings,mark=at position 0.4 with {\drawx}}}},
	kernel/.style={semithick,shorten >=1pt,shorten <=1pt,->},
	multx/.style={shorten >=1pt,shorten <=1pt,
		postaction={decorate,decoration={markings,mark=at position 0.5 with {\drawx}}}},
	kernelx/.style={semithick,shorten >=1pt,shorten <=1pt,->,
		postaction={decorate,decoration={markings,mark=at position 0.4 with {\drawx}}}},
	kernel1/.style={->,semithick,shorten >=1pt,shorten <=1pt,postaction={decorate,decoration={markings,mark=at position 0.45 with {\draw[-] (0,-0.1) -- (0,0.1);}}}},
	kernel2/.style={->,semithick,shorten >=1pt,shorten <=1pt,postaction={decorate,decoration={markings,mark=at position 0.45 with {\draw[-] (0.05,-0.1) -- (0.05,0.1);\draw[-] (-0.05,-0.1) -- (-0.05,0.1);}}}},
	kernelBig/.style={semithick,shorten >=1pt,shorten <=1pt,decorate, decoration={zigzag,amplitude=1.5pt,segment length = 3pt,pre length=2pt,post length=2pt}},
	rho/.style={dotted,semithick,shorten >=1pt,shorten <=1pt},
	renorm/.style={shape=circle,fill=white,inner sep=1pt},
	labl/.style={shape=rectangle,fill=white,inner sep=1pt},
cumu2n/.style={inner sep=3pt},
cumu2/.style={draw=red!50,fill=red!20},
cumu3/.style={regular polygon, regular polygon sides=3,draw=red!50,rounded corners=3pt,fill=red!20,minimum size=5mm},
cumu4/.style={regular polygon, regular polygon sides=4,draw=red!50,rounded corners=3pt,fill=red!20,minimum size=7mm},
cumu5/.style={regular polygon, regular polygon sides=5,draw=red!50,rounded corners=3pt,fill=red!20,minimum size=5mm},
Ip/.style={draw=symbols},
xi/.style={circle,fill=symbols!10,draw=symbols,inner sep=0pt,minimum size=1.5pt},
xix/.style={crosscircle,fill=symbols!10,draw=symbols,inner sep=0pt,minimum size=1.2mm},
xib/.style={circle,fill=symbols!10,draw=symbols,inner sep=0pt,minimum size=1.6mm},
xibx/.style={crosscircle,fill=symbols!10,draw=symbols,inner sep=0pt,minimum size=1.6mm},
not/.style={circle,fill=symbols,draw=symbols,inner sep=0pt,minimum size=0.5mm},
I/.style={draw=symbols,
		decorate, decoration={zigzag,amplitude=0.4pt,segment length = 0.6pt,pre length=0.5pt,post length=0.5pt}},
>=stealth,
}
\makeatletter
\def\DeclareSymbol#1#2#3{\expandafter\gdef\csname MH@symb@#1\endcsname{\tikz[baseline=#2,scale=0.15,draw=symbols]{#3}}\expandafter\gdef\csname MH@symb@#1s\endcsname{\scalebox{0.7}{\tikz[baseline=#2,scale=0.15,draw=symbols]{#3}}}}
\def\<#1>{\csname MH@symb@#1\endcsname}
\makeatother

\DeclareSymbol{X}{-2.4}{\node[xi] {};}
\DeclareSymbol{1}{0}{\draw[white] (-.5,0) -- (.5,0); \draw[Ip] (0,0) node[dot] {} -- (0,1.25); \node[xi] at (0,1.4) {};}
\DeclareSymbol{d}{0}{\draw[white] (-.5,0) -- (.5,0); \draw[I] (0,0) node[dot] {} -- (0,1.25); \node[xi] at (0,1.4) {};}
\DeclareSymbol{2}{0}{\node(root)[dot] at (0,0) {}; \draw[Ip] (-0.5,1.4) node[xi] {} -- (root); \draw[Ip] (0.5,1.4) node[xi] {} -- (root);}

\DeclareSymbol{20}{-3}{\node(root)[dot] at (0,0) {}; \draw[Ip] (root) -- (0,-1) node[dot] {};\draw[Ip] (-.7,1) node[xi] {} -- (root); \draw[Ip](.7,1) node[xi] {} -- (root);}
\DeclareSymbol{2d}{-3}{\node(root)[dot] at (0,0) {}; \draw[I] (root) -- (0,-1) node[dot] {};\draw[Ip] (-.7,1) node[xi] {} -- (root); \draw[Ip] (.7,1) node[xi] {}--  (root);}
\DeclareSymbol{21}{-3}{\node(root)[dot] at (0,-1) {}; \draw[Ip] (-1,1) node[xi] {} -- (-0.5,0); \draw[Ip] (root) -- (-0.5,0); \draw[Ip] (root) -- (1,0) node[xi] {};\draw[Ip] (-0.5,0) node[dot] {} -- (0.5,1) node[xi] {};}

\DeclareSymbol{40}{-3}{\node(right)[dot] at (1.2,0) {};\node(left)[dot] at (-0.4,0) {};\node(root)[dot] at (0.4,-1) {};\draw[Ip] (-.8,1) node[xi] {} -- (left); \draw[Ip] (0,1) node[xi] {} -- (left);\draw[Ip] (.8,1) node[xi] {} -- (right); \draw[Ip] (right) -- (1.6,1) node[xi] {};\draw[Ip] (left) -- (root); \draw[Ip] (right) -- (root);}

\DeclareSymbol{211}{2}{\draw[Ip] (-1.2,2.4) node[xi] {} -- (-0.8,1.6)   node[dot] {} -- (-0.4,0.8)   node[dot] {} -- (0,0)  node[dot] {} -- (0.8,0.8) node[xi] {};\draw[Ip] (-0.4,0.8) -- (0.4,1.6) node[xi] {};\draw[Ip] (-0.8,1.6) -- (0,2.4) node[xi] {};}

\DeclareSymbol{210}{2}{\draw[Ip] (-1.2,2.4) node[xi] {} -- (-0.8,1.6)  node[dot] {} -- (-0.4,0.8)  node[dot] {} -- (0,0)  node[dot] {};\draw[Ip] (-0.4,0.8) -- (0.4,1.6) node[xi] {};\draw[Ip] (-0.8,1.6) -- (0,2.4) node[xi] {};}
\DeclareSymbol{21d}{2}{\draw[Ip] (-1.2,2.4) node[xi] {} -- (-0.8,1.6)  node[dot] {} -- (-0.4,0.8); \draw[I] (-0.4,0.8) node[dot] {} -- (0,0)  node[dot] {};\draw[Ip] (-0.4,0.8) -- (0.4,1.6) node[xi] {};\draw[Ip] (-0.8,1.6) -- (0,2.4) node[xi] {};}

\DeclareSymbol{10}{1}{\draw[Ip] (0,2) node[xi] {} -- (-1,1) node[dot] {} -- (0,0) node[dot] {};}
\DeclareSymbol{1d}{1}{\draw[Ip] (0,2) node[xi] {} -- (-1,1); \draw[I] (-1,1) node[dot] {} -- (0,0) node[dot] {};}
\DeclareSymbol{11}{1}{\draw[Ip] (0,2) node[xi] {} -- (-0.8,1) node[dot] {} -- (0,0) node[dot] {} -- (0.8,1) node[xi] {};}

\newcommand{\mft}{\mathfrak{t}}
\newcommand{\mfl}{\mathfrak{l}}

\def\s{\mathfrak{s}}

\newcommand{\e}{\varepsilon}

\def\DD{\mathscr{D}}

\def\${|\!|\!|}

\def\Z{\mathbb{Z}}

\def\R{\mathbb{R}}
\def\N{\mathbb{N}}

\let\wtd\widetilde

\def\t{\mathfrak{t}}

\definecolor{Red}{rgb}{1,0,0}
\definecolor{Blue}{rgb}{0,0,1}
\definecolor{Olive}{rgb}{0.41,0.55,0.13}
\definecolor{Yarok}{rgb}{0,0.5,0}
\definecolor{Green}{rgb}{0,1,0}
\definecolor{MGreen}{rgb}{0,0.8,0}
\definecolor{DGreen}{rgb}{0,0.55,0}
\definecolor{Yellow}{rgb}{1,1,0}
\definecolor{Cyan}{rgb}{0,1,1}
\definecolor{Magenta}{rgb}{1,0,1}
\definecolor{Orange}{rgb}{1,.5,0}
\definecolor{Violet}{rgb}{.5,0,.5}
\definecolor{Purple}{rgb}{.75,0,.25}
\definecolor{Brown}{rgb}{.75,.5,.25}
\definecolor{Grey}{rgb}{.7,.7,.7}
\definecolor{Black}{rgb}{0,0,0}

\begin{document}

\title{Local solution to the multi-layer KPZ equation}
\author{Ajay Chandra$^1$ and Dirk Erhard$^2$ and Hao Shen$^3$}
\institute{Imperial College London, UK, \email{achandra@imperial.ac.uk} 
\and Universidade Federal da Bahia, Brazil, \email{erharddirk@gmail.com}
\and University of Wisconsin-Madison, US, \email{pkushenhao@gmail.com}}

\maketitle

\begin{abstract}
In this article we prove local well-posedness of the system of equations 
$\partial_t h_{i}= \sum_{j=1}^{i}\partial_x^2 h_{i}+ (\partial_x h_{i})^2 + \xi $ on the circle
 where
$1\leq i\leq N$ and 
$\xi$ is a space-time white noise. 
We attempt to generalize the renormalization procedure which gives the Hopf-Cole solution for the single layer equation and our $h_1$ (solution to the first layer) coincides with this solution. However, we observe that cancellation of logarithmic divergences that occurs at the first layer does not hold at higher layers and develop explicit combinatorial formulae for them. 
\end{abstract}

\setcounter{tocdepth}{2}
\tableofcontents

\section{Introduction}
\label{S1}

The aim of this paper is to introduce a system of equations called the multi-layer KPZ equation, formally given by
\begin{equ}  [e:MKPZintro]
\partial_t h_i 
=  \sum_{j=1}^{i}  \partial_x^2 h_j + (\partial_x h_i)^2 
	+ \xi   \qquad i\in\{1,2,\cdots, N\} \;,
\end{equ}
where $\xi$ is a space-time white noise.
We
study the local well-posedness and renormalization  of
this system on the circle.
 This system is interesting for  several reasons.

First, there has been increasing interest recently in the study of coupled KPZ systems. In \cite{funaki2016coupled}, the local solution to the following coupled system is studied
\begin{equ}  \label{e:Funaki-KPZ}
\partial_t h^\alpha = \frac12 \partial_x^2 h^\alpha 
	+ \frac12  \sum_{\beta,\gamma=1}^n  
		\Gamma_{\beta\gamma}^\alpha  \partial_x h^\beta \partial_x h^\gamma 
	+ \sum_{\beta=1}^n \sigma^\alpha_\beta \xi^\beta
\qquad \alpha\in\{1,2,\cdots, N\} \;.
\end{equ}
This is an $\R^n$ vector valued KPZ system driven by $n$ independent space-time white noises, with constant coefficients $\Gamma_{\beta\gamma}^\alpha$ and $\sigma^\alpha_\beta$. 
When $n=1$, this then becomes  the KPZ equation.
In \cite{hairer2016motion}, Hairer studied the equation describing the random motion of a string in an $n$ dimensional manifold driven by $m$ independent space-time white noises, written as
\begin{equ}  \label{e:Hairer-String}
\partial_t h^\alpha =  \partial_x^2 h^\alpha 
	+  \sum_{\beta,\gamma=1}^N  
		\Gamma_{\beta\gamma}^\alpha (h)  \partial_x h^\beta \partial_x h^\gamma 
	+ \sum_{i=1}^m \sigma^\alpha_i (h) \xi_i
\qquad \alpha\in\{1,2,\cdots, N\} \;,
\end{equ}
where $\Gamma_{\beta\gamma}^\alpha (h)$ are  the Christoffel symbols for the Levi-Civita connection, and  $\{\sigma_i (h)  \}$ is a collection of vector fields on the manifold.
In these works, the coupling between the equations in the system
arises from the nonlinearities. 
In the present article, we investigate a  new  situation, that is, one has a {\it matrix} in front of the diffusive term. In particular, the matrix is the lower triangle matrix with all the lower triangular entries being $1$. Moreover, the whole system is driven by only {\it one single} white noise $\xi$.
The system is thus such that the first $m$ equations evolve independently of the $m'$-th equation for all $m'>m$.
As a matter of fact, one could also think of the $i$-th equation
as a usual KPZ equation driven by $\xi$ together with  
``noises" $\partial_x^2 h_1,\cdots,\partial_x^2 h_{i-1}$.

Another motivation of studying the present problem, especially the particular choice of the lower triangular diffusive matrix, is the following.
In the paper \cite{NeilJon} the authors introduced a system of (continuum) multi-layer partition functions $Z_N$ for $N \in \N$
in order to better understand the integrability structure behind the discrete or semi-discrete polymer models, and their connection with quantum Toda lattice and geometric RSK correspondence:
\[
Z_N(t,x,y)=p(t,x,y)^N \Big( 1+\sum_{k=1}^\infty \int R^n_k ((t_1,x_1),\cdots,(t_k,x_k))
\xi^{\otimes k} (dt_1,\cdots,dx_k) \Big)
\]
where $R_k^n$ is the k-point correlation for a collection
of $n$ non-intersecting Brownian bridges starting at $x$ at time $0$
and ending at $y$ at time $t$,
and the integrals are stochastic iterated integrals
against the white noise $\xi$
over $(x_1,\cdots,x_k)\in \R^k$
and $0<t_1<\cdots <t_k<t$.
 It was shown in \cite{NeilJon} that this series representation for each $Z_n$ is  $L^2$ convergent. In particular, when $N=1$, $R_k$ is the transition density
of the Brownian motion passing through the points $(t_1,x_1),\cdots,(t_k,x_k)$, and
 the above   series   becomes
the chaos series solution to the stochastic heat equation $Z=\frac12 \partial_y^2 Z + Z\xi$ with Dirac initial data.
 The processes $Z_N$ are related with a system of stochastic PDEs in the following way. Define $u=u_1=Z_1$ and for $n\ge 2$ define $u_n=Z_n/Z_{n-1}$. Then
 pretending that $\xi$ is smooth, \cite{NeilJon} derived via a Karlin-McGregor formula
 and Darboux transformation 
  a coupled system of stochastic (nonlinear) heat equations for  $u_n$:
 \footnote{In  \cite{NeilJon} the processes depend on two spatial parameters $x$ and $y$, but here we set $x$ to be zero and write $x$ for $y$ for notational simplicity.}
\begin{equ} [e:MSHE]
\partial_t u_n = \frac12 \partial_x^2 u_n + \Big(\xi 
	+ \partial_x^2 \log(Z_{n-1}/p^{n-1})\Big)\,u_n
\end{equ}
with initial conditions $u_n(0,x)=\delta(x)$.
Then  in \cite[Proof of Prop.~3.3 and Prop.~3.7]{NeilJon}, the authors introduced 
$h_n = \log u_n$ which formally satisfy
\[
\partial_t h_n = \frac12 \partial^2_x h_n + \frac12 (\partial_x h_n)^2 + \xi
+ \partial_x^2 \log (Z_{n-1}/p^{n-1}) \;.
\]
A direct computation shows that 
\[
\partial_x^2 \log (1/p(x,t)^n) = -n\partial_x^2 \log p(x,t)
=-n \cdot \frac{(\partial_x^2 p) p - (\partial_x p)^2}{p^2}
=n/t
\] 
So we arrive at the equation
$
\partial_t h_n = \frac12 \partial^2_x h_n 
	+ \frac12 (\partial_x h_n)^2 + \xi
+  \sum_{j=1}^{n-1} \partial_x^2 h_{j} + \frac{n-1}{t} $.
Of course, the system \eqref{e:MKPZintro} we study in this article
is slightly different from the setting in \cite{NeilJon};
for instance we are
on a finite interval with periodic boundary condition instead of the full line,
and we will not consider Dirac initial data, and
we  also drop the term $ \frac{n-1}{t}$ (which arises from the Dirac initial condition).

Before we formulate our main result we fix an even, smooth, compactly supported function $\rho:\R^2\to\R$, integrating to one and we set for $\eps\in (0,1]$,
\begin{equ} [e:mollifier]
\rho_\eps(t,x)= \eps^{-3}\rho(\eps^{-2}t,\eps^{-1}x),\quad \mbox{and}\quad \xi_\eps=\xi*\rho_\eps.
\end{equ}

Let $S^1$ be the unit circle and $\CC^\beta$ be the space of H\"older
continuous functions defined as in \cite{Regularity}.
Our main result is the following.

\begin{theorem}
\label{thm}
Fix $ N\in\N$ and $(h_1^{(0)},\cdots,h_N^{(0)})\in (\CC^\beta(S^1))^{\otimes N}$ with $\beta\in(0,1)$.
For each $i\in\{1,2,\ldots, N\}$, there exists a sequence of constants 
$(C_{\varepsilon,i}:1\le i \le N)$  tending to infinity as $\varepsilon$ tends to zero, such that the solutions $(h_{\eps,1},\cdots,h_{\eps,N})$
to the sequence of systems of $N$ equations
\begin{equ} \label{eq:MKPZ}
\partial_t h_{\varepsilon,i} 
= \sum_{j=1}^{i} \partial_x^2 h_{\varepsilon,j} 
+ (\partial_x h_{\varepsilon,i})^2 + \xi_{\varepsilon}-C_{\varepsilon,i} 
\end{equ}
starting from initial condition $(h_1^{(0)},\cdots,h_N^{(0)})$
converge in probability to $N$ random fields $(h_1,\cdots,h_N)\in\CC^\eta([0,T],S^1)^{\otimes N}$
for any $\eta\in(0,\beta\wedge \frac12)$. 
Here, the constants $(C_{\varepsilon,i}:1\le i \le N)$ are such that
\begin{equs}
C_{\eps,1} & =\eps^{-1} C_1 +  c_1 \\
C_{\eps,i} & =\eps^{-1} C_i + \tilde C_i \log\eps + c_i  \qquad (2\le i\le N)
\end{equs}
for finite constants $C_i,c_i,\tilde C_j$ ($1\le i\le N$, $2\le j\le N$).
Moreover, if $C_1$ is chosen to suitably depend on $\rho$ (defined in \eqref{e:mollifier}) then the limit $(h_{1},\dots,h_{N})$ does not depend on the choice of the mollifier $\rho$ and 
$h_1$ agrees with the Hole-Copf solution of the KPZ equation driven by space-time white noise.
\end{theorem}

Recall that for the standard KPZ equation \cite{KPZ},
and the manifold-valued KPZ type equation \eqref{e:Hairer-String} in the geometric setting,
 and the vector valued KPZ system \eqref{e:Funaki-KPZ} under the so called ``trilinear" condition therein,  the logarithmic renormalization constants cancel out.
 In our case, we expect that $\tilde C_j$ in general do not vanish for $j\ge 2$, but we have not found a general argument to prove this (however, see Remark~\ref{rem:logs} for the first several layers.)
For the standard KPZ equation,
the cancellation of logarithmic renormalization is ``expected"
because the It\^o term appearing  from the It\^o calculus for the Cole-Hopf transform behaves as $1/\eps$,
and the linear stochastic heat equation is defined 
in It\^o sense.
For the multilayer KPZ however,
there is no Cole-Hopf transform that linearizes the system,
and the corresponding nonlinear  multilayer stochastic heat equation
 can not be simply defined in the It\^o sense.
Requiring  non-vanishing
 logarithmic renormalizations seems not surprising in our case.

We think that it may be of interest to study  a general type of KPZ system
\begin{equ} 
\partial_t h^\alpha = \frac12 \sum_{\beta=1}^\alpha
\nu_{\beta}^\alpha  \partial_x^2 h^\beta 
	+ \frac12  \sum_{\beta,\gamma=1}^n  
		\Gamma_{\beta\gamma}^\alpha  \partial_x h^\beta \partial_x h^\gamma 
	+ \sum_{\beta=1}^n \sigma^\alpha_\beta \xi^\beta
\qquad \alpha\in\{1,2,\cdots, N\} \;,
\end{equ}
where $\nu_{\alpha}^\alpha>0$ and $\nu_{\beta}^\alpha=0$ when $\beta>\alpha$. In particular, finding the condition for the coefficients under which the  logarithmic renormalizations  cancel would be very interesting.

\vspace{2ex}
{\bf Acknowledgement.}
A.Chandra gratefully acknowledges financial support from the Leverhulme Trust via an Early Career Fellowship, ECF-2017-226.
H.Shen gratefully acknowledges financial support by the NSF award DMS-1712684 and DMS-1909525.

\section{Notation and formulation}
\label{S2}

\emph{Notation.} 
We will denote by $\binom{j}{i}$ the binomial coefficients,
with convention $\binom{j}{i}=0$ when $i,j\ge 0$ and $j<i$.
We sometimes also need the case that $i$ or $j$ is $-1$: we have
$\binom{n}{-1}=0$ for $n\ge 0$ and $\binom{-1}{-1}=\binom{0}{0}=1$.

We start by rewriting \eqref{eq:MKPZ} in a mild formulation. Let 
\[
G(t,x)
\eqdef
\sum_{k \in 2 \pi \Z}
\frac{1}{\sqrt{4 \pi t}}
\exp \left[ - \frac{(x-k)^{2}}{4t} 
\right]
\]
for $t> 0$ and $G(t,x)=0$ for $t\le 0$
be the heat kernel on $S^{1}$. For each  $i \in \N$, we recursively define kernels $G_{i}$ by setting $G_{0} \eqdef G$ and for $i > 0$ setting $G_{i}
\eqdef
G \ast \partial_{x}^{2}G_{i-1}$, so that
\[
G_i = G * (\partial_{x}^{2}G)^{*i}
\]
where $*$ stands for convolution over space and time variables,
and for any function $H$ over space-time $H^{*i}$ stands for the convolution of $i$ copies of $H$.
We also define for $j \ge 1$, 
\begin{equ} [e:defBbar]
\overline{G}_{j} 
\eqdef  G* \partial_{x}^{2}G * (\delta + \partial_{x}^{2}G)^{*(j-1)}
= \sum_{i=1}^{j} \binom{j-1}{i-1} G_{i}
\end{equ}
where $\delta$ is the Dirac distribution, and for any function $F$ it is understood that $F^{*0} =\delta$. The last equality of \eqref{e:defBbar} follows by a binomial identity.
Moreover we define $\overline{G}_{0} \eqdef G$. For $\eps\in (0,1]$ 
and $\rho_\e$ as in \eqref{e:mollifier}, 
we  introduce the shorthands
\[
G_{i,\eps}= G_i*\rho_\eps \;,
\qquad
\overline{G}_{i,\eps}= \overline{G}_{i}*\rho_\eps\;.
\]
Furthermore, if $h$ is a function {\it only} defined in space, then with a slight abuse of notation we also write
\begin{equ}
[e:spaceconvolution]
(\overline{G}_j*h)(t,x)= \int \overline{G}_j(t,x-y)h(y)\, dy.
\end{equ}
We then have the following lemma.
\begin{lemma}\label{lem:mild}
For any $N \in \N$ the mild formulation of the (unrenormalized) system \eqref{eq:MKPZ} is given by
\begin{equ}\label{mildMKPZ}
h_{\varepsilon,j} = 
\sum_{i = 1}^{j}
\overline{G}_{j-i} \ast \left[ 
(\partial_{x} h_{\varepsilon,i})^{2}
+
\xi_{\varepsilon}
+h_{i}^{(0)}
\right] 
 \quad \textnormal{for }1 \le j \le N.
\end{equ}
\end{lemma}
\begin{proof}
We proceed by induction in $j$. The case $j=1$ is immediate since $\overline{G}_{0} = G$. 
Suppose that the Equation~\eqref{mildMKPZ} holds for all $j \le k$. 
Then by the induction assumption we have
\begin{equs}
{} & h_{\varepsilon,k+1}
=
G \ast \left[ 
(\partial_{x}h_{\varepsilon,k+1})^{2} +
\xi_\eps + h_{k+1}^{(0)}\right] 
+ G \ast
\sum_{j = 1}^{k} 
\partial_x^2 h_{\eps,j}
\\
&=
G \ast \left[ 
(\partial_{x}h_{\varepsilon,k+1})^{2} +
\xi_\eps + h_{k+1}^{(0)}\right] 
+ G \ast
\sum_{j = 1}^{k} 
\sum_{i = 1}^{j}
\partial_x^{2} \overline{G}_{j-i} \ast \left[ 
(\partial_{x} h_{\varepsilon,i})^{2}+
\xi_\eps + h_{i}^{(0)}\right]\\
&=
G \ast
\left[ 
(\partial_{x}h_{\varepsilon,k+1})^{2}+
\xi_\eps + h_{k+1}^{(0)}
\right]
+
\sum_{i = 1}^{k} 
\Big( \sum_{j = i}^{k}
G \ast \partial_x^{2}
\overline{G}_{j-i} \Big)
\ast \left[ 
(\partial_{x} h_{\varepsilon,i})^{2}
+\xi_\eps + h_{i}^{(0)}
\right]\;.
\end{equs}
Note  that for $1 \le i \le k$, one has by $\overline{G}_{0} = G$ and \eqref{e:defBbar}
\begin{equs}
\sum_{j = i}^{k}
G \ast \partial_x^{2}
\overline{G}_{j-i} 
& =
G*\partial_x^2 G 
+ \sum_{j = i+1}^{k}
G* (\partial_{x}^{2}G)^{*2} * (\delta + \partial_{x}^{2}G)^{*(j-i-1)}
\\
&=
G*\partial_x^2 G 
+G* \partial_{x}^{2}G * \Big((\delta+\partial_x^2 G)^{*(k-i)}-\delta\Big)
\\
&=\overline{G}_{k+1-i}
\end{equs}
where we used the following identity 
\[
(H-\delta)* \sum_{j = 0}^{m-1}  H^{*j} = H^{*m}-\delta
\]
with $H=\delta+\partial_x^{2} G$ and $m=k-i$. 
Therefore the statement holds for $k+1$.
\end{proof}
We fix for the remainder of the entire paper the number of layers $N \in \N$ in the statement of Theorem~\ref{thm}. 
 
\section{Applying the theory of regularity structures}
In this section we describe how we recast our problem in the language of the theory of regularity structures \cite{Regularity}. We will invoke the machinery of \cite{BHZ16} and \cite{CH16} to define a regularity structure and renormalization group rich enough to solve \eqref{mildMKPZ} and show convergence of appropriately renormalized models. 

\subsection{Construction of the multi-layer KPZ regularity structure}
\label{sec:MKPZ-RS}

We fix a parabolic space-time scaling $\s$ on $\R^{2}$ by setting $\s = (2,1)$. 

We introduce two finite, abstract, sets $\mfL_{+} \eqdef \{\mft_{(j,i)}:\ 1 \le i \le j \le N \}$ and $\mfL_{-} = \{\mfl\}$ which, respectively, index the regularizing kernels and driving noises $\mfl$ appearing in our problem. Here, $\mfl$ represents an occurrence of the driving noise $\xi$ and $\mft_{(j,i)}$ represents an occurrence of (a truncation of) the kernel $\overline{G}_{j-i}$. 
Clearly there is some redundancy in our labeling of kernels in the sense that if $j-i = j'-i' = k \in \N$ then $\mft_{(j,i)}$ and $\mft_{(j',i')}$ both represent an occurrence of the same kernel $\overline{G}_{k}$. However, for fixed $k$, the kernel $\overline{G}_{k}$ can appear in many different equations. 
Keeping track of these separate occurrences is convenient when defining our regularity structure when it comes to describing the set of allowable trees in terms of a \emph{rule}, thus we write $\mft_{(j,i)}$ to represent the specific occurrence of $\overline{G}_{j-i}$ in the equation for \emph{the }$j$\emph{-th layer equation}.

Next, we must specify a homogeneity assignment $|\cdot|_{\s}: \mfL_{-} \sqcup \mfL_{+} \rightarrow \R$. To that end we set $|\mfl|_{\s} \eqdef -\frac{3}{2} - \kappa$  and 
$|\mft|_{\s} \eqdef 2$ for any $\mft \in \mfL_{+}$.
Here $\kappa \in (0,\frac{1}{12})$ is a parameter in the construction of our regularity structure which is fixed throughout the paper.
\subsubsection{The multilayer KPZ rule and regularity structure}
\label{S:regstructure}
\noindent\textbf{Invoking the machinery of \cite{BHZ16}}:
The relevant reference for this section is \cite[Section 5]{BHZ16}.
We denote by $\mcb{E}$ and $\mcb{N}$ the sets of edge types and node types as defined in \cite[Eq (5.5)]{BHZ16}.
We adopt a shorthand for $\mcb{E}$, writing $\CI_{(j,i)}$ to represent $(\mft_{(j,i)},0) \in \mcb{E}$, $\CI'_{(j,i)}$ to represent $(\mft_{(j,i)},(0,1)) \in \mcb{E}$, and $\Xi$ to represent $(\mfl,0)$.  

We now define a rule $R: \mfL \rightarrow \mcb{PN} \setminus \{\emptyset\}$ as in \cite[Definition 5.7]{BHZ16}. 
For any $\mfl \in \mfL_{-}$ we set $R[\mfl] \eqdef \{()\}$. For $\mft_{(j,i)}\in \mfL_{+}$ we set
\[
R[\mft_{(j,i)}]
\eqdef
\{ (\CI'_{(i,k)}, \CI'_{(i,k')}): 1 \le k,k' \le i \}
\sqcup
\{ (\CI'_{(i,k)}) : 1 \le k \le i \}
\sqcup 
\{ (\Xi),() \}
\]
This rule is clearly normal and if we define $\mathrm{reg}: \mfL \rightarrow \R$ via $\mathrm{reg}(\mfl) \eqdef -3/2 - 2 \kappa$ for $\mfl \in \mfL_{-}$ and $\mathrm{reg}(\mft) \eqdef -1/2 - 3 \kappa$ for $\mft \in \mfL_{+}$ then we see that $R$ is subcritical in the sense of \cite[Definition 5.13]{BHZ16}.  

Finally, let $\bar{R}$ be the completion of $R$ as constructed in \cite[Proposition 5.20]{BHZ16}, and $\mathscr{T} = (\mcb{T},G)$ be the associated \emph{reduced} regularity structure truncated above homogeneity $\gamma \eqdef \frac{3}{2} + 3 \kappa$.\\

\noindent\textbf{The trees of our regularity structure.}
For those readers not familiar with \cite{BHZ16} we now try to give some more intuition for what our regularity structure looks like. 
Recall that a key ingredient of a regularity structure is a graded vector space and in most applications to SPDEs (including this paper) this graded vector space is a free vector space generated by a family of rooted decorated trees. We describe how to think of these trees in a certain formal symbolic notation where edges are seen as unary operators. 

The nodes of these trees correspond to symbols of the form $\Xi$ representing an occurrence of the driving noise or to $\mathbf{X}^{k}$ for some $k \in \N^{d}$ which represents a polynomial. 
Then, $\mathbf{X}^{2}\CI'_{(3,2)}[\Xi]\CI'_{(3,2)}[\Xi]$ is an example of a tree with three nodes (corresponding to $\mathbf{X}^{2}$, the root, and two nodes of the form $\Xi$) and two edges (the two instances of $\CI'_{(3,2)}$). 
Our ``rule'' described above just enforces that if there are $\CI$ or $\CI'$ operators inside some of the $\CI'_{(j,i)}[\cdot]$ or $\CI_{(j,i)}[\cdot]$ then they must be of the form $\CI'_{(i,k)}$. 
\subsection{Graphical Notation}\label{sec: graphical notation}
In order to obtain the form of the SPDE that the BPHZ renormalized model 
satisfies we need to understand the form of the solution $H = (H_{i})_{i =1}^{N}$. 
This requires explicitly working with the trees of our regularity structure.
Since $N$ is arbitrary the set of these trees can be arbitrarily large but this is only because $\mfL_{+}$ becomes larger. 
The \emph{shapes} of the trees that we see are always just those of the single layer KPZ equation and we now introduce a graphical notation that takes advantage of this.  

We introduce the symbolic trees 
\[
\{ \<1>, \<d>, \<2>, \<20>, \<2d>, \<21>, \<40>, \<211>, \<210>, \<21d>, \<10>, \<1d>, \<11> \}\;.
\]
These symbols will be used to refer to the shape of trees. The circular nodes represent occurrences of the noise $\Xi$, thicker lines correspond to an occurrence of $\CI_{\mft}$ and the thinner line to an instance of $\CI'_{\bar{\mft}}$.

Starting from one of the symbolic trees above, we get a tree in our regularity structure by assigning the specific labels $\mft_{(j,i)} \in \mfL_{+}$ to the edges of a tree in a way that respects the rule $\bar{R}$ that determines our regularity structure. 

To lighten notation, we will sometimes denote elements of $\mfL_{+}$ just by $(j,i)$, $1 \le i \le j \le N$, rather than $\mft_{(j,i)}$.
We then write expressions like 
\[
\<2>[(j_{1},i_{1}),(j_{2},i_{2})]
\]
 where $(j_{1},i_{1})$ specifies the line to the left and $(j_{2},i_{2})$ specifies the line to the right. 

We will use this graphical notation only for trees that obey the rule $R$ rather than the larger class that that obey the weaker rule $\bar{R}$. 
The rule $R$ implements a constraint on our indices, that is for $\<2>[(j_{1},i_{1}),(j_{2},i_{2})]$ to be a tree in our regularity structure we require $j_{1} = j_{2}$. 
To avoid having to write out constraints everywhere we enforce that in our notation a symbol with a set of indices not allowed by $R$ corresponds to $0$.

Note that our notation introduces some redundancy, one has $\<2>[(j_{1},i_{1}),(j_{2},i_{2})] = \<2>[(j_{2},i_{2}),(j_{1},i_{1})]$.
 However, this will turn out to be a feature rather than a bug as it will free us from keeping track of some symmetry factors in the formulas to follow. 

For larger trees we label edges \emph{from left to right}, \emph{going from the top to the bottom}. For instance, we have
\begin{equs}
\<40>[(j_{1},i_{1}),\dots,(j_{6},i_{6})]
=&
\CI'_{(j_{5},i_{5})}
\Big[
\CI'_{(j_{1},i_{1})}[\Xi] \CI'_{(j_{2},i_{2})}[\Xi]
\Big]
\CI'_{(j_{6},i_{6})}
\Big[
\CI'_{(j_{3},i_{3})}[\Xi]\CI'_{(j_{4},i_{4})}[\Xi]
\Big]\\
\<211>[(j_{1},i_{1}),\dots,(j_{6},i_{6})]
=&
\CI'_{(j_{5},i_{5})}
\Big[
\CI'_{(j_{3},i_{3})} 
\big[ \CI'_{(j_{1},i_{1})}[\Xi] \CI'_{(j_{2},i_{2})}[\Xi]
\big]
\CI'_{(j_{4},i_{4})}[\Xi]
\Big]
\CI'_{(j_{6},i_{6})}[\Xi]
\end{equs}
For the tree $\<40>$ the constraint on indices is 
\[
j_{3} = j_{4} = i_{6},\ j_{2} = j_{1} = i_{5},\ j_{5}=j_{6}\;.
\] 
For the tree $\<211>$ the constraint on indices is
\[
j_{1} = j_{2} = i_{3},\ j_{3} = j_{4} = i_{5},\ j_{5} = j_{6}\;.
\] 
We will often just write expressions like $\<40>[(j_{k},i_{k})_{k=1}^{6}]$ to save space.
When we write $(j_{k},i_{k})_{k=1}^{l}$ it is understood we are referring to a tuple of pairs $(j_{1},i_{1}),\dots,(j_{l},i_{l})$ with $1 \le i_{k} \le j_{k} \le N$ for each $1 \le k \le l$.

As a final example, for $\<2d>[(j_{k},i_{k})_{k=1}^{3}]$ the constraint on indices is $j_{1} = j_{2} = i_{3}$. 
\subsection{Models for the Multilayer KPZ Equation} \label{sec:The MKPZ Model}
We fix kernels $K_0,\ldots, K_{N-1}$ mapping $\R^2$ into $\R$ such that each $K_i$ (a) is compactly supported in the ball of  radius one around the origin, (b) is symmetric in space, (c) coincides with $\overline{G}_{i}$ on the ball of radius $1/2$, and (d) for every polynomial of $\s$-degree 2 one has
$\int_{\R^{2}} K_i(z)Q(z)\, dz=0$.

By~\cite[Lemma 5.5]{Regularity} and Proposition~\ref{Prop:G} the $K_i$'s can be additionally chosen such that one has the identity $\overline{G}_i=K_i+R_i$, for a smooth remainder $R_i$.

The assignment of kernels to elements of $\mfL_{+}$ is then given by $\mft_{(j,i)} \in \mfL_{+}$, $K_{\mft_{(j,i)}} \eqdef K_{j - i}$. 

We write $\mathscr{M}_{\infty}(\mathscr{T})$ and $\mathscr{M}_{0}(\mathscr{T})$ the spaces of smooth admissible models on $\mathscr{T}$ and its closure, respectively.

\begin{definition}\label{def: BPHZ lift}
Given any stationary $\mcb{C}(\R^{3})$-valued random variable $\zeta$ we denote by $Z^{\zeta}_{\BPHZ}$ the $\mathscr{M}_{\infty}(\mathscr{T})$-valued random variable called the BPHZ lift of $\zeta$, see \cite[Theorem 6.17]{BHZ16} for the definition of the BPHZ lift. 
\end{definition}
The following proposition
produces  the random admissible model that will allow us to describe a solution to the multilayer KPZ equation.

\begin{proposition} Let $\xi$ denote space-time white noise. The BPHZ lift admits an extension to $\xi$ in the following sense -- there exists a $\mathscr{M}_{0}(\mathscr{T})$-valued random variable $Z_{\BPHZ}^{\xi}$ such that for any smooth compactly supported function $\rho: \R^{2} \rightarrow \R$ with $\int_{\R^{2}} \rho(z) dz = 1$ one has that the random models $Z^{\xi \ast \rho_{\eps}}_{\BPHZ}$ converges in probability to $Z_{\BPHZ}^{\xi}$ on $\mathscr{M}_{0}(\mathscr{T})$ as $\eps \downarrow 0$ where $\rho_{\eps}$ is defined as in \eqref{e:mollifier}.
\end{proposition}
\begin{proof}
This proposition follows from \cite[Theorem 2.15]{CH16}. In addition to subcriticality \cite[Theorem 2.15]{CH16} enforces certain assumptions on the homogeneities of $\tau \in \mcb{T}$ with $\tau \neq \Xi$. 
In our setting all one must check is that for every such $\tau$ one has $|\tau|_{\s} > -|\s|/2$. The second is that any such $\tau$ must have $|\tau|_{\s} + |\Xi|_{\s} + |\s| > 0$. Both of these conditions guarantee that none of the random fields in our local expansions will have diverging variances. 

One can easily check that the worst such trees are those of the form $\CI_{(j,i)}'[\Xi]^{2}$ which have homogeneity $-1-2\kappa$. 
 We see that  $\kappa \in (0,1/6)$  suffices to guarantee the two conditions mentioned above.
\end{proof}
\begin{remark}
Note that the solution we refer to in Theorem~\ref{thm} will be a finite shift of the solution 	obtained from the BPHZ model.
\end{remark}

\subsection{Abstract Fixed Point Problem}
In view of Lemma~\ref{lem:mild} the fixed point problem associated to~\eqref{eq:MKPZ} can be written as
\begin{equ}\label{eq:abstractfp}
\vec{H}=\CP\big( \vec{(\DD H)^2} + \vec{\Xi} \big)+ P h^{(0)}.
\end{equ}
The above fixed point problem  should be seen as a vector fixed point problem where the components of vector are elements of an appropriate $\CD^{\gamma,\eta}$-space. It can  be compared to~\cite[Eq (2.12)]{NonGaussShen}. 
We now make it more precise by specifying the notation used aboved.  

$\CP$ is an $N$ by $N$ matrix of operators on a given $\CD^{\gamma,\eta}$ space.
Writing the matrix components $\CP = (\CP_{j,i})_{i,j=1}^{n}$ one has $\CP_{j,i} = 0$ for $j < i$ and for $j \ge i$ we set $\CP_{j,i}$ to be an abstract integration operator corresponding to $\overline{G}_{j-i}$. In particular, we use the decomposition $(K_{j-i},R_{j-i})$ to specify the local rough part and non-local smooth part of the operator - see \cite[Section 5]{Regularity}.
 Then $\vec{H}$, $\vec{(\DD H)^2}$ and $P\vec{h_{0}}$ are $N$ component vectors of elements of $\CD^{\gamma,\eta}$. 
  Writing $\vec{H} = (H_{i})_{i=1}^{N}$, $H_{i} \in D^{\gamma,\eta}$ is the modelled distribution describing the $i$-th layer of the equation. We then set 
  \[
  \vec{(\DD H)^2} 
  \eqdef
  \Big( 
  (\DD H_{i})^{2}
  \Big)_{i=1}^{N}
  \] 
  Finally, one has $Ph^{(0)} \eqdef (Ph^{(0)}_{i})_{i=1}^{N}$ where $h_{i}^{(0)}$ is the initial data for the $i$-th layer as referenced in Theorem~\ref{thm} and $Ph^{(0)}_{i}$ is the result of applying the heat kernel $G$ to this initial data and lifting it to $\CD^{\gamma,\eta}$ as in \cite[Lemma 7.5]{Regularity}.

We have the following lemma.

\begin{lemma} \label{lem:abs-fp-sol}
Choose $\gamma >3/2+\kappa$ and $\eta>0$ such that each $h_{i,0}\in \CC^{\eta}$. Then, for every admissible model, there exists $T>0$ such that the fixed point problem~\eqref{eq:abstractfp} has a unique solution in $\vec{H} \in \CD^{\gamma,\eta} ([0,T]\times S^1)$.
\end{lemma}
\begin{proof}
We can proceed by induction since $\CP_{j,i} = 0$ for $j < i$. The fixed point problem associated to $H_1$ (the first coordinate of $\vec{H}$) coincides with the fixed point problem of the KPZ equation, see also~\cite[Section 2]{NonGaussShen} and it follows for instance from~\cite[Theorem~7.8]{Regularity} that, for every admissible model, there exists $T_1>0$ such that the above fixed point problem for $H_1$ has
a unique solution in 
$\CD^{\gamma,\beta} ([0,T_1]\times S^1)$. 
With the fixed point problems for $H_1,\ldots, H_i$ solved, we can then again invoke~\cite[Theorem~7.8]{Regularity} to show that there is $T_{i+1}>0$ such that the fixed point problem for $H_{i+1}$ has a unique solution in
$\CD^{\gamma,\beta} ([0,T_{i+1}]\times S^1)$;
this is because the operator $\CP_{j,i}$ does enjoy the desired Schauder estimate by
Proposition~\ref{Prop:G}.
This finishes the argument with $T=\min_i T_i$. 
\end{proof}
In order to later derive renormalised equations it will be useful to know what the solution promised in Lemma~\ref{lem:abs-fp-sol} looks like.
\begin{lemma}
Let $\gamma \in (3/2 + \kappa, 3/2 + 2 \kappa)$, $\eta > 0$, and $\vec{H} = (H_{l})_{i=l}^{N} \in \CD^{\gamma,\eta}([0,T] \times S^{1})$ be the solution to the abstract fixed point problem \eqref{eq:abstractfp} promised by Lemma~\ref{lem:abs-fp-sol}. Then $\vec{H}$ is of the following form: for $1 \le l \le N$ one has 
\begin{equs}\label{eq: shape of abstract solution}
H_{l}(z)
\eqdef&\; 
H_{l,\mathbf{1}}(z)\mathbf{1}
+
H_{l,\mathbf{X_{1}}}(z)\mathbf{X}_{1}
+
\sum_{
(j_{1},i_{1})
\atop
j_{1} = l}
\<d>[(j_{1},i_{1})]\\
{}
&\;
+
\sum_{
((j_{k},i_{k}))_{k=1}^{3}
\atop
j_{3} = l}
\<2d>[(j_{k},i_{k})_{k=1}^{3}]
+
2
\sum_{
((j_{k},i_{k}))_{k=1}^{5}
\atop
j_{5} = l}
\<21d>[(j_{k},i_{k})_{k=1}^{5}]\\
{}&\;
+
2
H_{l,\mathbf{X_{1}}}(z)
\sum_{
((j_{k},i_{k}))_{k=1}^{2}
\atop
j_{2} = l}
\<1d>[(j_{k},i_{k})_{k=1}^{2}]\;.
\end{equs}
\end{lemma}
\begin{proof}
Our choice of $\gamma$ constraints the monomials that are kept in any $\CD^{\gamma,\eta}$ expansion. The fact that the solution $\vec{H}$ must be of this form can be verified by observing that it reproduces itself upon being inserted on the right hand side of \eqref{eq:abstractfp}. 
\end{proof}
\section{Renormalization}
\label{sec:renorm}
If $Z^{\eps} = (\Pi^{\eps},\Gamma^{\eps})$ is the canonical lift of $\zeta_{\eps}$ into a smooth model then the BPHZ lift of $\hat{Z}^{\eps} \eqdef (\hat{\Pi}^{\eps},\hat{\Gamma}^{\eps})$ of $\zeta_{\eps}$ can be described as a shift of $Z^{\eps}$ by an element of the renormalization group $\mathfrak{R}$, a group of linear operators from $\mcb{T}$ to itself that was characterized in \cite{BHZ16}. 

In terms of the algebraic theory of renormalization of regularity structures the situation here is essentially\footnote{The only difference being the the $(j,i)$ decorations.} the same as those found in previous treatments of the KPZ equation \cite{KPZJeremy,NonGaussShen}.   
 A full specification of the BPHZ model is unnecessary for our purposes, we will only need to understand its action on certain elements of the regularity structure present so that we can derive the renormalized equation.
\subsection{Definition of renormalization constants}
We first specify the renormalization constants of the BPHZ model. 
We set, for collections of indices compatible with the corresponding symbol,

\begin{equs}\label{eq: def of BPHZ renorm constants}
\ell_{\eps,\<2>}[(j_{k},i_{k})_{k=1}^{2}]
\eqdef&
\E
\Big[ 
\int_{(\R^{3})^{2}}
dz_{1} dz_{2} 
\prod_{p=1}^{2}
K'_{j_{p} - i_{p}}(-z_{p})\zeta_{\eps}(z_{p})
\Big]\\
\ell_{\eps,\<40>}[(j_{k},i_{k})_{k=1}^{6}]
\eqdef&
\E
\Big[ 
\int_{(\R^{3})^{2}}
dz_{5}\ dz_{6}
K'_{j_{5} - i_{5}} (-z_{5}) 
K'_{j_{6} - i_{6}} (-z_{6})\\
{}&
\qquad 
\Big(
\int_{(\R^{3})^{2}}
dz_{1} dz_{2} 
\prod_{p=1}^{2}
K'_{j_{p} - i_{p}}(z_{5} - z_{p})\zeta_{\eps}(z_{p})
\Big)\\
{}&
\qquad
\Big(
\int_{(\R^{3})^{2}}
dz_{3} dz_{5} 
\prod_{p=3}^{4}
K'_{j_{p} - i_{p}}(z_{6} - z_{p})\zeta_{\eps}(z_{p})
\Big)
\Big]\\
\ell_{\eps,\<211>}[(j_{k},i_{k})_{k=1}^{6}]
\eqdef&
\E
\Big[ 
\int_{(\R^{3})^{2}}
dz_{5}\ dz_{6}
K'_{j_{5} - i_{5}} (-z_{5}) 
K'_{j_{6} - i_{6}} (-z_{6}) \zeta_{\eps}(z_{6})\\
{}&
\qquad 
\int_{(\R^{3})^{2}}
dz_{3} dz_{5}
K'_{j_{3} - i_{3}}
(z_{5} - z_{3})
K'_{j_{4} - i_{4}}(z_{5} - z_{4})
\zeta_{\eps}(z_{4})\\
{}&
\qquad
\Big(
\int_{(\R^{3})^{2}}
dz_{1} dz_{2} 
\prod_{p=1}^{2}
K'_{j_{p} - i_{p}}(z_{3} - z_{p})\zeta_{\eps}(z_{p})
\Big)
\Big]
\end{equs}
Here, for $j \ge 0$, $K_{j}'(t,x) \eqdef \partial_{x} K_{j}(t,x)$ where $K_{j}$ is as in Section~\ref{sec:The MKPZ Model}. 

We also define $\bar{\ell}_{\eps,\<2>}[(j_{k},i_{k})_{k=1}^{2}]$, $\bar{\ell}_{\eps,\<40>}[(j_{k},i_{k})_{k=1}^{6}]$, and $\bar{\ell}_{\eps,\<211>} [(j_{k},i_{k})_{k=1}^{6}]$ to be defined analogously as to the constants in \eqref{eq: def of BPHZ renorm constants} except every instance of $K'_{j}$, $j \ge 0$ should be replaced by $\bar{G}'_{j}$. 

We again adopt the convention that $\ell_{\eps,\<2>}$, $\ell_{\eps,\<40>}$, $\ell_{\eps,\<211>}$, $\bar{\ell}_{\eps,\<2>}$, $\bar{\ell}_{\eps,\<40>}$, $\bar{\ell}_{\eps,\<211>}$ vanish if given a set of indices that do not satisfy the constraint coming from the rule $R$. 

The constants of \eqref{eq: def of BPHZ renorm constants} are those used to define the BPHZ model but this choice is not completely canonical since what we call the BPHZ model depends on the choice of kernel truncation $K_{j}$ made in Section~\ref{sec:The MKPZ Model}. 

Using instead the renormalization constants defined in terms of $\bar{G}_{j}$ in our renormalised equation will yield the solution referred to in Theorem~\ref{thm}. However, the following lemma tells us that the discrepancy between the two sets of constants amounts to a finite shift. 
\begin{lemma}\label{lem: hopf cole vs bphz constants}
One has the convergence
\begin{equs}
\lim_{\eps \downarrow 0}
\bar{\ell}_{\eps,\<2>}[(j_{k},i_{k})_{k=1}^{2}]
-
\ell_{\eps,\<2>}[(j_{k},i_{k})_{k=1}^{2}] 
=&\ 
\bar{c}_{\<2>}[(j_{k},i_{k})_{k=1}^{2}] \\
\lim_{\eps \downarrow 0}
\bar{\ell}_{\eps,\<40>}[(j_{k},i_{k})_{k=1}^{6}]
-
\ell_{\eps,\<40>}[(j_{k},i_{k})_{k=1}^{6}]
=&\ \bar{c}_{\<40>}[(j_{k},i_{k})_{k=1}^{6}] \\
\lim_{\eps \downarrow 0}
\bar{\ell}_{\eps,\<211>}[(j_{k},i_{k})_{k=1}^{6}]
-
\ell_{\eps,\<211>}[(j_{k},i_{k})_{k=1}^{6}]
=&\ \bar{c}_{\<211>}[(j_{k},i_{k})_{k=1}^{6}]\;,
\end{equs}
where the three constants on the right hand sides above are all finite and we restrict the indices $((j_{k},i_{k})_{k=1}^{2})$ or $((j_{k},i_{k}))_{k=1}^{6}$ are chosen to be compatible with the corresponding symbol. 
\end{lemma}
\begin{proof}
For the discrepancy between the constants $\bar{\ell}_{\eps,\<2>}[(j_{k},i_{k})_{k=1}^{2}]$ and $\ell_{\eps,\<2>}[(j_{k},i_{k})_{k=1}^{2}]$ one can show that by scale invariance of $\bar{G}_{m}$ under parabolic scaling one has 
\[
\bar{\ell}_{\eps,\<2>}[(j_{k},i_{k})_{k=1}^{2}]
=
\frac{1}{\eps}\bar{\ell}_{1,\<2>}[(j_{k},i_{k})_{k=1}^{2}]\;.
\]
Then using that the kernels $K_{j_{1} - i_{1}}$ and $K_{j_{2} - i_{2}}$ are just truncations of $\bar{G}_{j_{1} - i_{1}}, \bar{G}_{j_{2} - i_{1}}$ one can show  
\[
\ell_{\eps,\<2>}[(j_{k},i_{k})_{k=1}^{2}]
=
\frac{1}{\eps}\bar{\ell}_{\eps,\<2>}[(j_{k},i_{k})_{k=1}^{2}]
+
c
\]
for some finite constant $c$. 

One needs more detailed analysis for the discrepancy  for the symbols $\<40>$ and $\<211>$ since these in general produce logarithmic divergences. 
However, very similar analysis was done in the context of a single layer KPZ equation in \cite[Section~6.3]{KPZJeremy}. 
There one can show that appropriate rescalings of the kernels $K_{0} \ast \rho$ and $\partial_{x} K_{0} \ast \rho$ always converge to the corresponding untruncated kernels $G_{0} \ast \rho$ and $\partial_{x} G_{0} \ast \rho$ in semi-normed spaces of kernels that give good control over convolution on large and small scales - see \cite[Lemma~6.8]{KPZJeremy} - and this can be used to control the error in renormalization constants introduced by choosing the truncation $K_{0}$ of $G_{0}$ in the integrals defining $\ell_{\eps,\<40>}[(0,0),\dots,(0,0)]$ and $\ell_{\eps,\<211>}[(0,0),\dots,(0,0)]$ - see \cite[Lemma~6.11]{KPZJeremy}. 

However, it is straightforward to check that one has convergence, under the same rescaling, of the  kernels $K_{j} \ast \rho$ and $\partial_{x} K_{j} \ast \rho$ to $\bar{G}_{j} \ast \rho$ and $\partial_{x} \bar{G}_{j} \ast \rho$ and this combined with the same arguments as before gives the desired result. 
\end{proof}
\subsection{Renormalization of symbols and renormalised equation}
%

\begin{lemma}\label{lem: renormalisation of symbols}
Let $Z^{\eps} = (\Pi^{\eps},\Gamma^{\eps})$ be the canonical lift of $\zeta_{\eps}$ into a smooth model and $\hat{Z}^{\eps} \eqdef (\hat{\Pi}^{\eps},\hat{\Gamma}^{\eps})$ be the BPHZ lift of $\zeta_{\eps}$. 

Then one has
\begin{equs}
\hat{\Pi}^{\eps}_{z} \mathbf{1}
=&
\Pi^{\eps}_{z} \mathbf{1}\\
\hat{\Pi}^{\eps}_{z} \<1>[(j_{1},i_{1})]   
=& \Pi^{\eps}_{z} \<1>[(j_{1},i_{1})] \\
\hat{\Pi}^{\eps}_{z} \<2>[(j_{k},i_{k})_{k=1}^{2}]  
=& \Pi^{\eps}_{z} \<2>[(j_{k},i_{k})_{k=1}^{2}]  -
\ell_{\eps,\<2>}[[(j_{k},i_{k})_{k=1}^{2}]  \\
\hat{\Pi}^{\eps}_{z} \<40>[(j_{k},i_{k})_{k=1}^{6}]  
=& \Pi^{\eps}_{z} \<40>[(j_{k},i_{k})_{k=1}^{6}]  
 - \ell_{\eps,\<40>}[(j_{k},i_{k})_{k=1}^{6}]\\
\hat{\Pi}^{\eps}_{z} \<211>[(j_{k},i_{k})_{k=1}^{6}]  
=& \Pi^{\eps}_{z} \<211>[(j_{k},i_{k})_{k=1}^{6}]  
 - \ell_{\eps,\<211>}[(j_{k},i_{k})_{k=1}^{6}]\;,\\
\end{equs}
where all the indices above are chosen to satisfy the constraint of the corresponding symbol. 
\end{lemma}
\begin{proof}
Starting with the generic formulation of BPHZ renormalization as a starting point, the key observations needed to justify the simple formulas above are that
\begin{itemize}
\item Additional subtractions for symbols of the type $\<40>$ or $\<211>$ due to the occurrence of $\<2>$ vanish since our truncated kernel $K$ integrates to $0$. 
\item We are in the Gaussian case, and so symbols with an odd number of noises do not produce any renormalization\footnote{In contrast to the situation in \cite{NonGaussShen}.}.
\item Symbols of the type $\<11>$ produce vanishing renormalization constants, this is by parity under the reflection $(t,x) \mapsto (t,-x)$ (since we imposed that our mollifier is even in space $\zeta_{\eps}$ is invariant in law under this reflection). 
\end{itemize}
\end{proof}
%
%
The next lemma presents the renormalized equation obtained by solving the equation driven by $\hat{Z}^{\eps}$.

\begin{lemma}\label{lem: BPHZ solution}
Let $\hat{Z}^{\eps}$ be the BPHZ lift of $\zeta_{\eps}$. Let $\gamma > 3/2 + \kappa$, $\eta > 0$, and $\vec{H}_{\eps}$ be the solution to the abstract fixed point problem \eqref{eq:abstractfp} in the $\CD^{\gamma,\eta}([0,T] \times S^{1})$ space over $\hat{Z}^{\eps}$. 
Let $\vec{u}_{\eps} = (u_{\eps,l})_{l=1}^{N}$ be the reconstruction of $\vec{H}_{\eps}$. 

Then $\vec{u}_{\eps}$ solves, on $[0,T] \times S^{1}$, the initial value problem
\begin{equ} \label{eq:MKPZ BPHZ}
\partial_t u_{\varepsilon,l} = \partial_x^2 u_{\varepsilon,l} + (\partial_x u_{\varepsilon,l})^2 + \xi_{\varepsilon}-\sum_{q=1}^{3} \tilde{C}^{(q)}_{\varepsilon,l} + \sum_{j=1}^{l-1} \partial_x^2 u_{\varepsilon,j}
\end{equ}
starting from initial condition $(h_1^{(0)},\cdots,h_N^{(0)})$, where 
\begin{equs}\label{eq: BPHZ constant}
\tilde{C}^{(1)}_{\varepsilon,l}
\eqdef&\ 
\sum_{
(j_{k},i_{k})_{k=1}^{2}
\atop
j_{1} = j_{2} = l}
\ell_{\eps,\<2>}[[(j_{k},i_{k})_{k=1}^{2}]\\
\tilde{C}^{(2)}_{\varepsilon,l}
\eqdef&\ 
\sum_{(j_{k},i_{k})_{k=1}^{6}
\atop
j_{5} = j_{6} = l}
\ell_{\eps,\<40>}[[(j_{k},i_{k})_{k=1}^{6}]\\
\tilde{C}^{(3)}_{\varepsilon,l}
\eqdef&\ 
4
\sum_{(j_{k},i_{k})_{k=1}^{6}
\atop
j_{5} = j_{6} = l}
\ell_{\eps,\<211>}[[(j_{k},i_{k})_{k=1}^{6}]\;.
\end{equs}
Moreover, the $(u_{1,\eps},\dots,u_{N,\eps})$ converge in probability to $N$ random fields $(u_1,\cdots,u_N)\in\CC^\eta([0,T],S^1)^{\otimes N}$
for any $\eta\in(0,\beta\wedge \frac12)$. 
\end{lemma}
\begin{proof}
Writing $\hat{\mathcal{R}}^{\eps}$ for the reconstruction operator associated to $\hat{Z}^{\eps}$, applying $\hat{\mathcal{R}}^{\eps}$ to both sides of \eqref{eq:abstractfp} gives
\begin{equ}\label{eq: reconstructed mildMKPZ}
u_{\varepsilon,j} = 
\sum_{i = 1}^{j}
\overline{G}_{j-i} \ast \left[ 
\hat{\mathcal{R}}^{\eps}[(\DD H_{i})^{2}]
+
\xi_{\varepsilon}
+h_{i}^{(0)}
\right] 
 \quad \textnormal{for }1 \le j \le N.
\end{equ}
Now, using the identity $(\hat{\mathcal{R}}^{\eps}F)(z) = (\hat{\Pi}_{z}F(z))(z)$, one can verify, using Lemma~\ref{lem: renormalisation of symbols}, 
\[
(\hat{\mathcal{R}}^{\eps}(\DD H_{i})^{2})(z) =
(\partial_{x} u_{\eps,i})^{2}(z) -  \sum_{q=1}^{3} \tilde{C}^{(q)}_{\eps,i}\;.
\]
Then after inserting the above into \eqref{eq: reconstructed mildMKPZ} one can rewrite the resulting system in the form \eqref{eq:MKPZ BPHZ} (just as \eqref{mildMKPZ} was a rewriting itself of the non-renormalized equation). 

The statement of convergence as $\eps \downarrow 0$ follows from the convergence of the models $\hat{Z}^{\eps}$, the continuity of the abstract solution map for \eqref{eq:abstractfp}, and the continuity of reconstruction.
\end{proof}
The above proposition gives a notion of local solution to the multilayer KPZ equation but it is not completely canonical since it depends on a truncation of the heat kernel and it will not give the Hopf-Cole solution for the first layer so it differs from the solution described in our main theorem but only by a finite shift.
\begin{proof}[Theorem~\ref{thm}]
For $i \in [N]$, $q=1,2,3$, we choose $C_{\e,i}^{(q)}$ to be defined as $\tilde{C}_{\e,i}^{(q)}$ but with $\ell_{\bullet}[\bullet]$ replaced by $\bar{\ell}_{\bullet}[\bullet]$. 
Convergence of the $\vec{h}_{\eps}$ then follows from that of the $\vec{u}_{\eps}$ of \eqref{lem: BPHZ solution} since one has
\[
h_{\eps,i}(t,x) - u_{\eps,i}(t,x)
=
t \sum_{q=1}^{3} (\tilde{C}_{\e,i}^{(q)} - C_{\e,i}^{(q)})
\] 
and the right hand side above converges to a finite limit by Lemma~\ref{lem: hopf cole vs bphz constants}.

Identification of $h_{1}$ with the corresponding Hopf-Cole solution can be performed just as in the proof of \cite[Proposition~7.1]{KPZJeremy} in the simpler setting where one just takes $F(u) = u^{2}$ there.  

The solution $h$ constructed there is obtained by taking renormalized solutions 
where $\sum_{q=1}^{3} \tilde{C}_{\e,1}^{(q)}$ is replaced by $C_{\e,1}^{(1)} + \tilde{C}_{\e,i}^{(2)} + \tilde{C}_{\e,i}^{(3)}$, then the constant $c$ referenced in \cite[Proposition~7.1]{KPZJeremy} corresponds precisely to $(\tilde{C}_{\e,i}^{(2)} + \tilde{C}_{\e,i}^{(3)}) - (C_{\e,i}^{(2)} + C_{\e,i}^{(3)})$. 
\end{proof}
We now proceed to obtain the promised combinatorial formula for the renormalization constants appearing in Theorem~\ref{thm}. 

\begin{proposition}
\label{prop:reneq}
$C_{\e,i}^{(1)}$ is given by \eqref{eq:1overeps} below,
$C_{\e,i}^{(2)} $ is given by \eqref{e:C2-formula}
where $ \mathcal G_{m_3,m_5,k_1,k_2}^{(\e)}$ therein
is given by \eqref{e:TallTreeCase1}--\eqref{e:TallTreeCase3B},
and $C_{\e,i}^{(3)} $ is given by \eqref{e:C3-formula}
where $\bar{\mathcal G}_{m_5;k_1,k_2;m_6}^{(\e)} $
 therein
is given by \eqref{e:WideTree}.
\end{proposition}


Note that although formulae for $ C_{\e,i}^{(2)}$ and $C_{\e,i}^{(3)}$
are very sophisticated,
they are all written in terms of integrals of Hermite polynomials and heat kernels (with many combinatoric coefficients).

We consider the constant  $C_{\e,i}^{(1)}$
in this section, and will study $ C_{\e,i}^{(2)}$ and $C_{\e,i}^{(3)}$
in the next section.

Consider $C_{\e,i}^{(1)}$.
Using the definition \eqref{e:defBbar}
 one has
\[
\bar\ell_{\eps,\<2>}[(i,j),(i,k)]
 =\sum_{m_1=0}^{i-j} \sum_{m_2=0}^{i-k} 
 \binom{i-j-1}{m_1-1}  \binom{i-k-1}{m_2-1} 
 G'_{m_1,\eps} * \wtd{G'}_{m_2,\eps} (0)
\]
where tilde denotes the reflection, i.e. 
$\wtd F(z):=F(-z)$.
In view of \eqref{eq: BPHZ constant} we need
to sum over $j$ and $k$. To this end
we perform a re-summing, using 
\begin{equs}
\sum_{j= 1}^{i}
\sum_{m=0}^{i-j} \binom{i-j-1}{m-1} H_m 
=  \sum_{m=0}^{i-1}   \sum_{j=1}^{i-m} 
	\binom{i-j-1}{m-1} H_m  
=   \sum_{m=0}^{i-1}  \binom{i-1}{m} H_m 
\end{equs}
where $H_m$ is a generic quantity which depends on $m$ but not $j$.
We get
\begin{equ}[e:1st-formula-Wick]
C_{\e,i}^{(1)}  
= \sum_{m_1,m_2=0}^{i-1}  \binom{i-1}{m_1}\binom{i-1}{m_2} 
G'_{m_1,\eps} * \wtd{G'}_{m_2,\eps} (0)
\end{equ}
Note that here we have used our convention for binomial constants with entries allowed to be $-1$.

\begin{lemma}\label{lem:1overeps}
Fix $i\in\{1,2,\ldots,N\}$. For each $k\in\{0,1,\ldots,i-1\}$,
with constant 
\[
C(\eps,k) =\frac{1}{k!} \big(t^k G^{(2k)} \big) *\rho_\eps^{(2)}(0) \simeq 1/\eps
\]
 one has
\begin{equation}  \label{eq:1overeps}
C_{\e,i}^{(1)}=\sum_{k=0}^{i-1}(-2)^k T(i-1,k)C(\eps,k),
\end{equation}
where for each $i$ and each $k$,
\begin{equation}
T(i,k)= \sum_{m_2=0}^{i}\sum_{m_1=k}^{i} (-2)^{-(m_1+m_2)}\binom{i}{m_1}\binom{i}{m_2}\binom{m_1-k+m_2}{m_2}.
\end{equation}
\end{lemma}

\begin{proof}
	Let
	 tilde denote the reflection, i.e. 
$\wtd F(z):=F(-z)$. Recall \eqref{e:1st-formula-Wick}.
We claim that
\begin{equs}\label{eq:1ststepC1}
G'_{m_1,\eps} & * \wtd{G'}_{m_2,\eps} (0)
\\
&=
-\sum_{k=0}^{m_1}(-2)^{-(m_1+m_2-k+1)} \binom{m_1+m_2-k}{m_2}
	(\wtd{G_k}*\rho_\eps^{(2)})(0)
\\
&-\sum_{k=0}^{m_2}(-2)^{-(m_1+m_2-k+1)}\binom{m_1+m_2-k}{m_1}(G_k*\rho_\eps^{(2)})(0),
\end{equs}
where $\rho_\eps^{(2)}=\rho_\eps*\rho_\eps$. 
Indeed, let $m_1,m_2\geq 0$, we  define
\begin{equation}
D_{m_1,m_2}^{(\e)}= G'_{m_1,\eps}*\wtd{G'}_{m_2,\eps},
\end{equation}
and by \eqref{e:def-Dij}
we see that $D_{m_1,m_2}^{(\e)} = D_{m_1,m_2} * \rho_\eps^{(2)}$.
Using \eqref{e:Dij-formula} we conclude the proof of~\eqref{eq:1ststepC1}.
%
%

We define
\begin{equation} \label{e:def-Cek}
C(\eps,k)\eqdef (G_k*\rho_\eps^{(2)})(0) = (\wtd{G_k}*\rho_\eps^{(2)})(0)
\end{equation}
where the last equality is a consequence of the fact that $\rho$ is an even function.
Using the form of $C_i^{(1)}$ stated in Proposition~\ref{prop:reneq}
\begin{equation}
C_{i}^{(1)}= \sum_{m_1,m_2=0}^{i-1}\binom{i-1}{m_1}\binom{i-1}{m_2}\ell_{1,\t},
\end{equation} 
together with \eqref{eq:1ststepC1}, we have 
\begin{equ}
C_i^{(1)}=\sum_{k=0}^{i-1}(-2)^k T(i-1,k) C(\eps,k),
\end{equ}
where for each $i$ and each $k$,
\begin{equation}
T(i,k)= \sum_{m_2=0}^{i}\sum_{m_1=k}^{i} (-2)^{-(m_1+m_2)}\binom{i}{m_1}\binom{i}{m_2}\binom{m_1-k+m_2}{m_2}.
\end{equation}
Here we have performed a resummation:
we  fix the $k$ in \eqref{eq:1ststepC1}, 
and sum $m_1$ from $k$ to $i-1$
(or sum $m_2$ from $k$ to $i-1$, which yields
the same result by \eqref{e:def-Cek} and invariance of swapping $m_1$ and $m_2$ - thus we get a factor $2$ which cancels a factor $-(-2)^{-1}$).

Moreover, integrating for each of the convolutions that follow first over space, making use of the semi-group property, and then integrating over time we see that
\begin{equation}
G^{*n}= \frac{t^{n-1}}{(n-1)!} G.
\end{equation}
Hence, we can then write 
\[
C(\eps,k)= (G_k*\rho_\eps^{(2)})(0)
 =\frac{1}{k!} \Big(t^k G^{(2k)} \Big) *\rho_\eps^{(2)}(0)\;.
\]
It now follows immediately that $C(\eps,k)\sim \eps^{-1}$.
\end{proof}

\begin{remark}
One could also derive an alternative formula for $C(\eps,k)$. Let
\begin{equation}
H_n(x)=(-1)^n e^{x^2/2}\partial_x^n e^{-x^2/2}
=: (-1)^n e^{x^2/2}\partial_x^n \tilde H(x)
\end{equation}
be the $n$-the Hermite polynomial where $ \tilde H(x) = e^{-x^2/2}$.
It then follows that
\begin{equation}
\partial_x^n G(x,t) =(-1)^n e^{x^2/4t}\partial_x^n \tilde H(x/\sqrt{2t})(-1)^n G(x,t)= \frac{1}{(2t)^{n/2}} (-1)^n H_n(x/\sqrt{2t})G(x,t).
\end{equation}
Hence, we can  write 
$C(\eps,k) 
= \frac{t^k}{k!}  \frac{1}{(2t)^k} F_{k} G *\rho_\eps^{(2)} (0)
= \frac{1}{2^k k!} (F_{k} G *\rho_\eps^{(2)})(0)$
where  $F_k(x,t)=H_{2k}(x/\sqrt{2t})$. 
\end{remark}

\section{Logarithmic renormalization}

We derive formulae for 
the other two constants in Proposition~\ref{prop:reneq} and
thus complete the proof of this Proposition.

\subsection{Combinatoric coefficients, and doing contractions}

Now consider the $n$-th equation. 
Using \eqref{e:defBbar} for $\bar G$, 
the first tree in the description of $h_n$, which we denote by $h_{n}^{(0)} $, is equal to
\begin{equs}[e:hn0]
 \sum_{i=1}^{n} \bar G_{n-i} \Xi
&= \sum_{i=1}^{n} \sum_{\ell=0}^{n-i} \binom{n-i-1}{\ell-1} 
	  G_{\ell} \xi 
=\sum_{\ell=0}^{n-1}   \sum_{i=1}^{n-\ell} 
	\binom{n-i-1}{\ell-1} \<1>[\ell] 
	\\
&= \sum_{\ell=0}^{n-1}  \binom{n-1}{\ell}   \<1>[\ell] 
\end{equs}
Note that $\sum_{\ell=0}^{n-1}$ can be rewritten as $\sum_{\ell\ge 0}$ since for $\ell$ out of that range the binomial coefficients are understood (with usual convention) as $0$.
%
The next tree in the description of $h_n$, denoted by $h_{n}^{(1)}$, is equal to
\begin{equs}[e:hn1]
 \sum_{i=1}^{n}& \bar G_{n-i} (\partial_x h_{i}^{(0)})^2 
= \sum_{i=1}^{n} \sum_{\ell=0}^{n-i} \binom{n-i-1}{\ell-1} 
	 G_{\ell}  \sum_{m_1,m_2=0}^{i-1}  \binom{i-1}{m_1}   \binom{i-1}{m_2}  \<2>[m_1,m_2]
	  \\
&= \sum_{\ell,m_1,m_2=0}^{n-1} 
	\Bigg\{\sum_{i=m_1 \vee m_2+1}^{n-\ell} 
	\binom{n-i-1}{\ell-1} \binom{i-1}{m_1}   \binom{i-1}{m_2}\Bigg\}
	\<20>[m_1,m_2,\ell] \\
\end{equs}
where we have performed resummation, 
\footnote{The coefficient in the parenthesis $\{\cdots\}$ 
automatically restricts the admissible labels for trees $\<20>[m_1,m_2,\ell]$.
For example, $n=2$ (2nd layer), if $\ell=1$ then $m_1,m_2$ must both be $0$, otherwise this coefficient is zero. 
} 
and as above  $\sum_{i=m_1 \vee m_2+1}^{n- \ell} $ could be  replaced by 
$\sum_{i\ge 0}$ since for $i$ out of that range the binomial coefficients are understood as $0$.


Turning to 
the next tree in the description of $h_n$, by \eqref{e:hn0} and \eqref{e:hn1}
\begin{equs}
 h_{n}^{(2)} &= 2 \sum_{i=1}^{n} \bar G_{n-i} (\partial_x h_{i}^{(0)} \partial_x h_{i}^{(1)}) \\
&= 2 \sum_{i=1}^{n} \sum_{\ell=0}^{n-i} \binom{n-i-1}{\ell-1}  G_{\ell} \!\!\!\!\sum_{m_1,m_2,m_3,m_4=0}^{i-1} 
	  	\sum_{j=m_1 \vee m_2+1}^{i-m_3} 
	    \\
	   &\qquad\qquad \binom{j-1}{m_1}   \binom{j-1}{m_2} \binom{i-1}{m_4} 
		 \binom{i- j -1}{ m_3-1}
	 \<21>[m_1,m_2,m_3,m_4] \\
&=2  \!\!\!\!\sum_{\ell,m_1,m_2,m_3,m_4=0}^{n-1} 
	c_{m_1,m_2,m_3,m_4,\ell}^{\<210>}
	\<210>[m_1,m_2,m_3,m_4,\ell] 
\end{equs}
where  we have performed resummation  and
\begin{equs}
{}&c_{m_1,m_2,m_3,m_4,\ell}^{\<210>} \eqdef    \\
{} & \sum_{i=\max(m_1,m_2,m_3,m_4)+1}^{n-\ell}  
		 \sum_{j\ge 0} 
	\binom{n-i-1}{\ell-1} \binom{j-1}{m_1}   \binom{j-1}{m_2}  \binom{i-1}{m_4} 
		 \binom{i- j -1}{ m_3-1} 
\end{equs}

So in the expansion of the right hand of the $n$-th layer equation, we have the following term
\begin{equs}[e:RHS211]
4 \sum_{m_1,\cdots,m_6=0}^{n-1} 
c_{m_1,m_2,m_3,m_4,m_5,m_6}^{\<211>}
\<211>[m_1,m_2,m_3,m_4,m_5,m_6]
\end{equs}
where
\begin{equs}
c_{m_1,m_2,m_3,m_4,m_5,m_6}^{\<211>}\eqdef
\sum_{i\in I}
		 \sum_{j\ge 0} 
	\binom{n-i-1}{m_5-1} \binom{j-1}{m_1}   \binom{j-1}{m_2}  \binom{i-1}{m_4} 
		 \binom{i- j -1}{ m_3-1} \binom{n-1}{ m_6} 
\end{equs}
where $I\eqdef [\max(m_1,m_2,m_3,m_4)+1,n-m_5]\cap \Z$.
In particular  when $n=1$, $c_{0,0,0,0,0,0}^{\<211>}=1$ 
which is consistent with the expansion for single-layer KPZ equation.

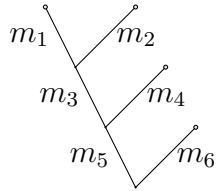
\begin{center}
\begin{tikzpicture}
\draw[Ip] (-1.2,2.4) node[xi] {} -- node[left] {$m_1$} (-0.8,1.6)   node[dot] {} -- node[left] {$m_3$}  (-0.4,0.8)   node[dot] {} -- node[left] {$m_5$}  (0,0)  node[dot] {} -- node[right] {$m_6$}   (0.8,0.8) node[xi] {};
\draw[Ip] (-0.4,0.8) -- node[right] {$m_4$}  (0.4,1.6) node[xi] {};
\draw[Ip] (-0.8,1.6) -- node[right] {$m_2$}  (0,2.4) node[xi] {};
\end{tikzpicture}
\end{center}

Now we do contractions for the noises.
 There are two ways of contraction (see the above figure): the first way is to contract the pair $(m_1,m_6)$ and the pair
$(m_2,m_4)$; the second way is  to contract the pair $(m_1,m_4)$ and  the pair $(m_2,m_6)$.

Write 
\[
G_k^{(\e)}\eqdef G_k *\rho_\eps^{(2)}\;.
\]
For the first contraction one has
\[
D_{m_1,m_6}^{(\e)}
=
-\sum_{k =0}^{m_6}
(-2)^{-(m_1+m_6 - k +1)}
\binom{m_1+m_6 - k}{m_1} G_k^{(\e)}  + ...
\]
\[
D_{m_2,m_4}^{(\e)}
=
-\sum_{k =0}^{m_4}
(-2)^{-(m_2+m_4- k +1)}
\binom{m_2+m_4 - k}{m_2} G_k^{(\e)}+ ...
\]
For the second contraction one has
\[
D_{m_1,m_4}^{(\e)}
=
-\sum_{k =0}^{m_4}
(-2)^{-(m_1+m_4 - k +1)}
\binom{m_1+m_4 - k}{m_1} G_k^{(\e)} + ...
\]
\[
D_{m_2,m_6}^{(\e)}
=
-\sum_{k =0}^{m_6}
(-2)^{-(m_2+m_6- k +1)}
\binom{m_2+m_6 - k}{m_2} G_k^{(\e)} + ...
\]
Here ``$...$'' stand for  terms with
 $\wtd{G_k}$; they are omitted, because they only contribute a finite constant to the contracted graphs. Indeed, with a reflected $G_k$, one has a contracted graph which contains a directed loop, for instance
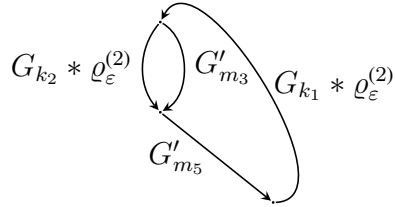
\[
\begin{tikzpicture}[scale=1.5]
\draw[kernel,bend right=40,->]  (0,1.6)   node[dot] {} to node[left] {$G_{k_2}*\rho_\e^{(2)}$}  (0,0.8)   node[dot] {};
\draw[kernel,->]  (0,0.8)  to node[left] {$G'_{m_5}$}  (1,0)  node[dot] {};
\draw[kernel,bend right=60,<-] (0,0.8) to node[right] {$G'_{m_3}$} (0,1.6);
\draw[kernel,bend right=120,->] (1,0) to node[right] {$G_{k_1}*\rho_\e^{(2)}$}  (0,1.6);
\end{tikzpicture}
\]
In this situation since $\rho_\eps^{(2)}$ is supported in a region of size of order $\eps$, and the heat kernel vanishes at negative time variables, the  integral corresponding to the graph is finite as $\e\to 0$.

Putting together these two ways of contractions and \eqref{e:RHS211}, one has
\begin{equs}
4\!\!\!\!\!\!
\sum_{m_1,\cdots,m_6=0}^{n-1} 
\sum_{i\in I,j\ge 0} 
	\binom{n-i-1}{m_5-1} \binom{j-1}{m_1}   \binom{j-1}{m_2}  \binom{i-1}{m_4} 
		 \binom{i- j -1}{ m_3-1} \binom{n-1}{ m_6}  \\
\times   \sum_{k_1 =0}^{m_6} \sum_{k_2 =0}^{m_4}
	(-2)^{-(m_1+m_2+m_4+m_6- k_1-k_2 +2)}
	\Big[
	\binom{m_1+m_6 - k_1}{m_1}
	\binom{m_2+m_4 - k_2}{m_2} \\
	+
	\binom{m_1+m_4 - k_2}{m_1}
	\binom{m_2+m_6 - k_1}{m_2}
	\Big]\\
\times \mathcal G_{m_3,m_5,k_1,k_2}^{(\e)}
\end{equs}
where
\begin{equ}[e:def-mathcalG]
 \mathcal G_{m_3,m_5,k_1,k_2}^{(\e)} \;\; \eqdef\;\;
\begin{tikzpicture}[scale=1.5,baseline=20]
\draw[kernel,bend right=40,->]  (0,1.6)   node[dot] {} to node[left] {$G_{k_2}^{(\e)}$}  (0,0.8)   node[dot] {};
\draw[kernel,->]  (0,0.8)  to node[left] {$G'_{m_5}$}  (1,0)  node[dot] {};
\draw[kernel,bend right=60,<-] (0,0.8) to node[right] {$G'_{m_3}$} (0,1.6);
\draw[kernel,bend right=120,<-] (1,0) to node[right] {$G_{k_1}^{(\e)}$}  (0,1.6);
\end{tikzpicture}
\end{equ}
A resummation yields
\begin{equs}
4\!\!\!\!\!\!\sum_{m_3,m_5,k_1,k_2=0}^{n-1} 
& \Bigg\{
\sum_{m_1,m_2=0}^{n-1} \sum_{m_4=k_2}^{n-1} \sum_{m_6=k_1}^{n-1} 
(-2)^{-(m_1+m_2+m_4+m_6- k_1-k_2 +2)}\\
& \times\sum_{i\in I,j\ge 0} 
	\binom{n-i-1}{m_5-1} \binom{j-1}{m_1}   \binom{j-1}{m_2}  \binom{i-1}{m_4} 
		 \binom{i- j -1}{ m_3-1} \binom{n-1}{ m_6}  \\
&\times  
	\Big[
	\binom{m_1+m_6 - k_1}{m_1}
	\binom{m_2+m_4 - k_2}{m_2} 
	+
	\binom{m_1+m_4 - k_2}{m_1}
	\binom{m_2+m_6 - k_1}{m_2}
	\Big]
	\Bigg\}  \\
&\times \mathcal G_{m_3,m_5,k_1,k_2}^{(\e)}
	\label{e:C2-formula}
\end{equs}

Now we consider the trees $\<40> [m_1,\cdots,m_6]$.
By \eqref{e:hn1}, in the expansion of the right hand of the n-th layer equation, we have the following
term
\begin{equs}
\sum_{m_1,\cdots,m_6=0}^{n-1} 
	\sum_{i=m_1 \vee m_2 +1}^{n-m_5} \sum_{j=m_3 \vee m_4 +1}^{n-m_6} 
	\binom{n-i-1}{m_5-1} \binom{i-1}{m_1} \binom{i-1}{m_2 } \\
	\binom{n-j-1}{m_6-1} \binom{j-1}{m_3} \binom{j-1}{m_4} 
\<40> [m_1,\cdots,m_6]
\end{equs}
where $i,j$ can equivalently sum over $i,j\ge 0$.

\[
\begin{tikzpicture}[scale=1.3]
\node(right)[dot] at (1.2,0) {};
\node(left)[dot] at (-0.4,0) {};
\node(root)[dot] at (0.4,-1) {};
\draw[Ip] (-.8,1) node[xi] {} -- node[left] {$m_1$} (left); 
\draw[Ip] (0,1) node[xi] {} -- node[right] {$m_2$} (left);
\draw[Ip] (.8,1) node[xi] {} -- node[left] {$m_3$} (right); 
\draw[Ip] (right) --node[right] {$m_4$} (1.6,1) node[xi] {};
\draw[Ip] (left) -- node[left] {$m_5$} (root); 
\draw[Ip] (right) -- node[right] {$m_6$} (root);
\end{tikzpicture}
\]

There are two ways of contraction (see the above figure for indices): the first is $(m_1,m_4)$ and
$(m_2,m_3)$; the other way is $(m_1,m_3)$ and $(m_2,m_4)$.
%
After contraction
\begin{equs}
\sum_{m_1,\cdots,m_6=0}^{n-1} 
	\sum_{i,j\ge 0}
	\binom{n-i-1}{m_5-1} \binom{i-1}{m_1} \binom{i-1}{m_2 } 
	\binom{n-j-1}{m_6-1} \binom{j-1}{m_3} \binom{j-1}{m_4} \\
\times  \Big\{ \sum_{k_1 =0}^{m_1} \sum_{k_2 =0}^{m_2}
	(-2)^{-(m_1+m_2+m_3+m_4- k_1-k_2 +2)}
	\Big[
	\binom{m_1+m_4 - k_1}{m_4}
	\binom{m_2+m_3 - k_2}{m_3} \\
	+
	\binom{m_1+m_3 - k_1}{m_3}
	\binom{m_2+m_4 - k_2}{m_4}
	\Big] \times\bar{\bar{\mathcal G}}_{m_5;k_1,k_2;m_6}^{(\e)}\\
+
 \sum_{k_1 =0}^{m_4} \sum_{k_2 =0}^{m_3}
	(-2)^{-(m_1+m_2+m_3+m_4- k_1-k_2 +2)}
	\Big[
	\binom{m_1+m_4 - k_1}{m_1}
	\binom{m_2+m_3 - k_2}{m_2} \\
	+
	\binom{m_1+m_3 - k_2}{m_1}
	\binom{m_2+m_4 - k_1}{m_2}
	\Big] \times \bar{\mathcal G}_{m_5;k_1,k_2;m_6}^{(\e)} \Big\}
\end{equs}
where we have omitted the terms  leading to graphs
with ``directed loops'' which only contribute finite constants, as explained above, and 
\begin{equ}[e:bar-mathcalG]
\bar{\mathcal G}_{m_5;k_1,k_2;m_6}^{(\e)}
\;
\eqdef
\;
\begin{tikzpicture}[scale=1.5,baseline=-10]
\node(right)[dot] at (1.2,0) {};
\node(left)[dot] at (-0.4,0) {};
\node(root)[dot] at (0.4,-1) {};
\draw[bend left = 80, ->] (left) to  node[above] {$G_{k_1}^{(\e)}$} (right); 
\draw[ ->] (left) to node[below] {$G_{k_2}^{(\e)}$} (right); 
\draw[->] (left) to node[left] {$G'_{m_5}$} (root); 
\draw[->] (right) to node[right] {$G'_{m_6}$} (root);
\end{tikzpicture}
\qquad
\bar{\bar{\mathcal G}}_{m_5;k_1,k_2;m_6}^{(\e)}
\;
\eqdef
\;
\begin{tikzpicture}[scale=1.5,baseline=-10]
\node(right)[dot] at (1.2,0) {};
\node(left)[dot] at (-0.4,0) {};
\node(root)[dot] at (0.4,-1) {};
\draw[bend left = 80, <-] (left) to  node[above] {$G_{k_1}^{(\e)}$} (right); 
\draw[ <-] (left) to node[below] {$G_{k_2}^{(\e)}$} (right); 
\draw[->] (left) to node[left] {$G'_{m_5}$} (root); 
\draw[->] (right) to node[right] {$G'_{m_6}$} (root);
\end{tikzpicture}
\end{equ}
Now note that in the parenthesis $\{\cdots\}$ the first sum is equal to the second sum, if we switch
$m_1\leftrightarrow m_3, m_2\leftrightarrow m_4, m_5\leftrightarrow m_6, i\leftrightarrow j,k_1\leftrightarrow k_2$. 
We only consider the second sum (namely the one with $\bar{\mathcal G}_{m_5;k_1,k_2;m_6}^{(\e)}$), which will cause a factor $2$.
%
(Of course for $k_1\neq k_2$, swapping $k_1$ and $k_2$ results in the same graph $\bar{\mathcal G}_{m_5;k_1,k_2;m_6}^{(\e)}$, but we do not combine these identical graphs for now.)

Combining the two sums in the parenthesis $\{\cdots\}$ and
re-summing, one has
\begin{equs}
2\!\!\!\!\!\!\sum_{m_5,m_6,k_1,k_2=0}^{n-1}   &  \Bigg\{
\sum_{m_1,m_2=0}^{n-1} 
\sum_{m_3=k_2}^{n-1}\sum_{m_4=k_1}^{n-1} 
(-2)^{-(m_1+m_2+m_3+m_4- k_1-k_2 +2)}\\
&\sum_{i,j\ge 0}
	\binom{n-i-1}{m_5-1} \binom{i-1}{m_1} \binom{i-1}{m_2 } 
	\binom{n-j-1}{m_6-1} \binom{j-1}{m_3} \binom{j-1}{m_4} \\
	&\Big[
	\binom{m_1+m_4 - k_1}{m_1}
	\binom{m_2+m_3 - k_2}{m_2} 
	+
	\binom{m_1+m_3 - k_2}{m_1}
	\binom{m_2+m_4 - k_1}{m_2}
	\Big] \Bigg\}\\
&\qquad\times
\bar{\mathcal G}_{m_5;k_1,k_2;m_6}^{(\e)} 
	\label{e:C3-formula}
\end{equs}

\subsection{Computation of contracted graphs}

Recall that the $n$-th Hermite polynomial is $H_n(x)=(-1)^n e^{\frac{x^2}{2}} \partial_x^n e^{-\frac{x^2}{2}}$. Let $\tilde H(x)=e^{-\frac{x^2}{2}}$.
Note that $\tilde H(x/\sqrt{2t}) = e^{-\frac{x^2}{4t}}$ and 
$\partial_x^n \tilde H(x/\sqrt{2t}) =\frac{1}{(2t)^{n/2}} H^{(n)}(x/\sqrt{2t})$

We have the following formula for derivatives of heat kernel
\begin{equ} [e:Gder-to-HG]
G^{(n)} (t,x) \eqdef \partial_x^n G(t,x) =
\frac{(-1)^n}{(2t)^{n/2}}  H_n(x/\sqrt{2t}) G(t,x)
\end{equ}
Another general formula:
\begin{equ} [e:G_n-to-Gder]
G_n (t,x) = G \star G'' \star \cdots \star G'' = \partial_x^{2n} G\star\cdots\star G = \frac{t^n}{n!} G^{(2n)}
\end{equ}

We compute 
$ \mathcal G_{m_3,m_5,k_1,k_2}^{(\e)}$
as defined in 
\eqref{e:def-mathcalG}.

{\bf Case 1.}
We first consider the case $m_3 - k_2 \in \{-1,0\}$, which means 
\[
|2m_3 +1 -2k_2|=1.
\]

The two ``parallel'' edges $G'_{m_3}$ and $G_{k_2}^{(\e)}$ 
in \eqref{e:def-mathcalG}
can be dealt with as follows: by \eqref{e:G_n-to-Gder}
\begin{equs} 
G'_{m_3} G_{k_2} 
= 
	 \frac{t^{m_3}}{m_3 !} G^{(2 m_3 +1)}
	  \frac{t^{k_2}}{k_2 !} G^{(2k_2)}
= \frac12 \frac{t^{m_3+k_2}}{m_3 ! k_2 !} 
\partial_x \Big( (G^{(2m_3+1)\wedge (2k_2)})^2 \Big)
\end{equs}
where the last step used the assumption $|2m_3 +1 -2k_2|=1$.
Replacing  $G'_{m_3}\cdot G_{k_2}^{(\e)}$ in
$ \mathcal G_{m_3,m_5,k_1,k_2}^{(\e)}$
by 
\[
\frac12 \frac{t^{m_3+k_2}}{m_3 ! k_2 !} 
\partial_x \Big( (G^{(2m_3+1)\wedge (2k_2)})^2 *\rho_\e^{(2)} \Big)
\]
 only causes a finite difference and does not change
the logarithmically divergent part of   $ \mathcal G_{m_3,m_5,k_1,k_2}^{(\e)}$.

With an integration by parts, we shift the x-derivative in the above expression to $G_{k_1}^{(\e)}$ in \eqref{e:def-mathcalG}, which produces a negative sign. The pair of edges connected to the bottom vertex in \eqref{e:def-mathcalG} is then $D_{m_5,k_1}$. We have
\begin{equs}
D_{m_5,k_1} &= -\sum_{\ell=0}^{k_1} 
	(-2)^{-(m_5+k_1-\ell+1)}  \binom{m_5+k_1-\ell}{m_5} G_\ell
	+ ... \\
&= -\sum_{\ell=0}^{k_1} (-2)^{-(m_5+k_1-\ell+1)}  \binom{m_5+k_1-\ell}{m_5} \frac{t^\ell}{\ell !} G^{(2\ell)} + ...
\end{equs}
where the terms ``...'' are omitted because they will only contribute finite constants as explained above.
So 
\begin{equs}
 \mathcal G_{m_3,m_5,k_1,k_2}^{(\e)} =
 \int_{|t|\ge\eps^2} &
\frac12 \sum_{\ell=0}^{k_1} 
	(-2)^{-(m_5+k_1-\ell+1)} 
	 \binom{m_5+k_1-\ell}{m_5} \\
	& \cdot \frac{t^{\ell+k_2+m_3}}{\ell ! k_2 ! m_3 !} 
	 G^{(2\ell)} \cdot (G^{(2m_3+1)\wedge (2k_2)})^2
\,dtdx
+O(1)
\end{equs}
where functions in the integrand depend on $(t,x)$,
and we have replaced mollifiers by a small scale cutoff in the integral which only causes an $O(1)$ constant.
Using \eqref{e:Gder-to-HG},  
we have
\begin{equs}
t^{\ell+k_2+m_3} &
	 G^{(2\ell)} \cdot (G^{(2m_3+1)\wedge (2k_2)})^2 \\
&=\frac{t^{\ell+k_2+m_3}}{(2t)^\ell (2t)^{(2m_3+1) \wedge (2k_2)}} 
	 H_{2\ell}(x/\sqrt{2t}) H_{(2m_3+1)\wedge (2k_2)}(x/\sqrt{2t})^2  G(t,x)^3 \;.
\end{equs}
Note that the entire power of $t$ does not depend on whether $m_3-k_2=-1$ or $0$, i.e., the powers of $t$ cancel out.

%

{\bf Case 2.}
Let's turn to  the case of $m_3-k_2 \ge 1$. 
We first write the two ``parallel'' edges  in \eqref{e:def-mathcalG} as
\[
G_{m_3}' G_{k_2} 
= \frac{t^{m_3+k_2}}{m_3! k_2!} G^{(2m_3+1)} G^{(2k_2)}
\]
Now the idea is the following: if we integrate by parts w.r.t. the  vertex $v$ (see \eqref{e:IBP-case2}), a derivative will be shifted from $G^{(2m_3+1)}$
to either $G'_{m_5}$ or $G^{(2k_2)}$. The good situations
are that either we have $G''_{m_5}$ (one of these two derivatives can be shifted to $G_{k_1}$) so that we can 
apply the identity for $D_{\cdot,\cdot}$ to convolve the bottom vertex (as first term on RHS of \eqref{e:IBP-case2}),
or the two parallel edges represent derivatives of $G$
of orders that only differ by one (so that it can be written as 
$\frac12 \partial_x ((\cdots)^2)$ as in Case 1).
To be more precise, starting from \eqref{e:def-mathcalG},
and ignoring for the moment all powers of $t$, using~\eqref{e:G_n-to-Gder} the above leads to the analysis of
\begin{equ}[e:IBP-case2]
\begin{tikzpicture}[scale=1.5,baseline=20]
\node at (-0.2,0.8) {$v$};
\draw[kernel,bend right=40,->]  (0,1.6)   node[dot] {} to node[left] {$G^{(2k_2)}$}  (0,0.8)   node[dot] {};
\draw[kernel,->]  (0,0.8)  to node[left] {$G'_{m_5}$}  (1,0)  node[dot] {};
\draw[kernel,bend right=20,<-] (0,0.8) to node[right] {$G^{(2m_3+1)}$} (0,1.6);
\draw[kernel,bend right=130,<-] (1,0) to node[right] {$G_{k_1}$}  (0,1.6);
\end{tikzpicture}
 = 
\begin{tikzpicture}[scale=1.5,baseline=20]
\draw[kernel,bend right=40,->]  (0,1.6)   node[dot] {} to node[left] {$G^{(2k_2)}$}  (0,0.8)   node[dot] {};
\draw[kernel,->]  (0,0.8)  to node[left] {$G''_{m_5}$}  (1,0)  node[dot] {};
\draw[kernel,bend right=20,<-] (0,0.8) to node[right] {$G^{(2m_3)}$} (0,1.6);
\draw[kernel,bend right=130,<-] (1,0) to node[right] {$G_{k_1}$}  (0,1.6);
\end{tikzpicture}
 -
\begin{tikzpicture}[scale=1.5,baseline=20]
\draw[kernel,bend right=40,->]  (0,1.6)   node[dot] {} to node[left] {$G^{(2k_2+1)}$}  (0,0.8)   node[dot] {};
\draw[kernel,->]  (0,0.8)  to node[left] {$G'_{m_5}$}  (1,0)  node[dot] {};
\draw[kernel,bend right=20,<-] (0,0.8) to node[right] {$G^{(2m_3)}$} (0,1.6);
\draw[kernel,bend right=130,<-] (1,0) to node[right] {$G_{k_1}$}  (0,1.6);
\end{tikzpicture}
\end{equ}
Here and below we omit the mollifiers to make the graphic notation simpler.
If $m_3-k_2\in\{-1,0\}$ we stop at this point
because the second term above can be studied as in Case 1, otherwise we apply another integration by parts to the second term above and we repeat this process until we can make use of the techniques used in the case $m_3-k_2\in\{-1,0\}$. In this way we see that the graph  becomes 
\begin{equation}
\begin{aligned}
&\begin{tikzpicture}[scale=1.5,baseline=20]
\draw[bend right=40,->]  (0,1.6)   node[dot] {} to node[left] {$G^{(2k_2)}$}  (0,0.8)   node[dot] {};
\draw[->]  (0,0.8)  to node[left] {$G''_{m_5}$}  (0.5,0)  node[dot] {};
\draw[bend right=20,<-] (0,0.8) to node[right] {$G^{(2m_3)}$} (0,1.6);
\draw[bend right=130,<-] (0.5,0) to node[right] {$G_{k_1}$}  (0,1.6);
\end{tikzpicture} -
\begin{tikzpicture}[scale=1.5,baseline=20]
\draw[bend right=40,->]  (0,1.6)   node[dot] {} to node[left] {$G^{(2k_2+1)}$}  (0,0.8)   node[dot] {};
\draw[->]  (0,0.8)  to node[left] {$G''_{m_5}$}  (0.5,0)  node[dot] {};
\draw[bend right=20,<-] (0,0.8) to node[right] {$G^{(2m_3-1)}$} (0,1.6);
\draw[bend right=130,<-] (0.5,0) to node[right] {$G_{k_1}$}  (0,1.6);
\end{tikzpicture}
+ 
\begin{tikzpicture}[scale=1.5,baseline=20]
\draw[bend right=40,->]  (0,1.6)   node[dot] {} to node[left] {$G^{(2k_2+2)}$}  (0,0.8)   node[dot] {};
\draw[->]  (0,0.8)  to node[left] {$G'_{m_5}$}  (0.5,0)  node[dot] {};
\draw[bend right=20,<-] (0,0.8) to node[right] {$G^{(2m_3-1)}$} (0,1.6);
\draw[bend right=130,<-] (0.5,0) to node[right] {$G_{k_1}$}  (0,1.6);
\end{tikzpicture}\\
&=\ldots=
\sum_{\ell=1}^{m_3-k_2}(-1)^{\ell-1}\begin{tikzpicture}[scale=1.5,baseline=20]
\draw[bend right=40,->]  (0,1.6)   node[dot] {} to node[left] {$G^{(2k_2+\ell-1)}$}  (0,0.8)   node[dot] {};
\draw[->]  (0,0.8)  to node[left] {$G''_{m_5}$}  (0.5,0)  node[dot] {};
\draw[bend right=20,<-] (0,0.8) to node[right] {$G^{(2m_3+1-\ell)}$} (0,1.6);
\draw[bend right=130,<-] (0.5,0) to node[right] {$G_{k_1}$}  (0,1.6);
\end{tikzpicture} +
(-1)^{m_3-k_2}\begin{tikzpicture}[scale=1.5,baseline=20]
\draw[bend right=40,->]  (0,1.6)   node[dot] {} to node[left] {$G^{(m_3+k_2)}$}  (0,0.8)   node[dot] {};
\draw[->]  (0,0.8)  to node[left] {$G'_{m_5}$}  (0.5,0)  node[dot] {};
\draw[bend right=20,<-] (0,0.8) to node[right] {$G^{(m_3+k_2+1)}$} (0,1.6);
\draw[bend right=130,<-] (0.5,0) to node[right] {$G_{k_1}$}  (0,1.6);
\end{tikzpicture} \\
&:= A + B.
\end{aligned}
\end{equation}

Now for the term $A$ we convolve the origin
\begin{equ}
D_{m_5,k_1} 
= -\sum_{s=0}^{k_1} (-2)^{-(m_5+k_1-s+1)}  \binom{m_5+k_1-s}{m_5} \frac{t^s}{s !} G^{(2s)} + \cdots
\end{equ}
where ``$\cdots$'' are again  terms that will only contribute finite constants.
For each fixed $s$ we need to integrate
\[
\frac{t^{s+m_3+k_2}}{s! m_3! k_2 !} 
G^{(2s)} G^{(2m_3+1-\ell)} G^{(2k_2+\ell-1)}
\]
Using \eqref{e:Gder-to-HG} we get
\[
\frac{1}{s! m_3! k_2 !} \frac{1}{2^{s+m_3+k_2}} 
\int_{|t|\ge \e^2} H_{2s}(x/\sqrt{2t}) H_{2m_3 + 1 -\ell} (x/\sqrt{2t})
H_{2k_2+\ell-1}(x/\sqrt{2t})
G(t,x)^3 \,dtdx
\]


Regarding $B$, note that 
$G^{(m_3+k_2+1)} G^{(m_3+k_2)} 
=\frac12 \partial_x ((G^{(m_3+k_2)} )^2)$. Shifting this $\partial_x$
to $G_{k_1}$ and integrating out the bottom vertex,
we can proceed as before. 


{\bf Case 3.}
We can apply the same strategy for the case $m_3-k_2 \le -2$. Integrate by parts until one can write the product of the parallel edges as $1/2$ times the spacial derivative of a square. 

Below we summarize the formulas we found.
\begin{itemize}
\item[1.] If $m_3-k_2\in\{-1,0\}$. Then the tall tree equals
\begin{equs}[e:TallTreeCase1]
\frac12 \sum_{s=0}^{k_1}&(-2)^{-(m_5+k_1-s+1)}{m_5+k_1-s\choose m_5}
\frac{1}{s!k_2!m_3!2^{s+(2m_3+1)\wedge(2k_2)}}\\
&\times\int H_{2s}(x/\sqrt{2t})H_{(2m_3+1)\wedge(2k_2)}(x/\sqrt{2t})^2 G(x,t)^3.
\end{equs}
\item[2.] If $m_3-k_2\geq 1$. In this case the tall tree can be written as $A+B$ with
\begin{equs}[e:TallTreeCase2A]
A= \sum_{\ell=1}^{m_3-k_2}&(-1)^{\ell+1}\sum_{s=0}^{k_1}(-2)^{-(m_5+k_1-s+1)}{m_5+k_1-s\choose m_5}
\frac{1}{s!m_3!k_2!2^{s+m_3+k_2}}\\
&\times\int H_{2s}(x/\sqrt{2t})H_{(2m_3+1-\ell)}(x/\sqrt{2t})H_{(2k_2+\ell-1)}(x/\sqrt{2t})G(x,t)^3
\end{equs}
and
\begin{equs}[e:TallTreeCase2B]
B= (-1)^{m_3-k_2}\frac12\sum_{s=0}^{k_1}&(-2)^{-(m_5+k_1-s+1)}{m_5+k_1-s\choose m_5}
\frac{1}{s!m_3!k_2!2^{s+m_3+k_2}}\\
&\times \int H_{2s}(x/\sqrt{2t})H_{m_3+k_2}(x/\sqrt{2t})^2G(x,t)^3.
\end{equs}
\item[3.] If $m_3-k_2\leq -2$ we should get that the tall tree is again of the form $A+B$ with
\begin{equs}[e:TallTreeCase3A]
A= \sum_{\ell=1}^{k_2-m_3-1}&(-1)^{\ell+1}\sum_{s=0}^{k_1}(-2)^{-(m_5+k_1-s+1)}{m_5+k_1-s\choose m_5}
\frac{1}{s!m_3!k_2!2^{s+m_3+k_2}}\\
&\times\int H_{2s}(x/\sqrt{2t})H_{(2m_3+\ell)}(x/\sqrt{2t})H_{(2k_2-\ell)}(x/\sqrt{2t})G(x,t)^3
\end{equs}
and 
\begin{equs}[e:TallTreeCase3B]
B= (-1)^{k_2-m_3-1}\frac12\sum_{s=0}^{k_1}&(-2)^{-(m_5+k_1-s+1)}{m_5+k_1-s\choose m_5}
\frac{1}{s!m_3!k_2!2^{s+m_3+k_2}}\\
&\times \int H_{2s}(x/\sqrt{2t})H_{m_3+k_2}(x/\sqrt{2t})^2G(x,t)^3.
\end{equs}
\end{itemize}

Let's describe a general formula for the logarithmic constant from the trees $\<40>$. 
Recall  $\bar{\mathcal G}_{m_5;k_1,k_2;m_6}^{(\e)}$ defined in
\eqref{e:bar-mathcalG}.

The two edges pointing to the bottom vertex carry the kernels $G'_{m_5}$ and $G'_{m_6}$.
This is precisely our kernel $D_{m_5,m_6}$. One has
\begin{equation}
D_{m_5,m_6}=
-\sum_{s=0}^{m_5}(-2)^{-(m_5+m_6-s+1)}{m_5+m_6-s\choose m_6}G_s +\cdots,
\end{equation}
where $\cdots$ is the part that only  contribute finite constants.
Fix $s$. Then, we need to evaluate $G_s G_{k_1}G_{k_2}$. Recall that $G_k(x,t)=\frac{t^k}{k!}G^{(2k)}(x,t)=\frac{t^k}{k!(2t)^k}H_{2k}(x/\sqrt{2t})G(x,t)$.
Thus, we get
\begin{equs}[e:WideTree]
\bar{\mathcal G}_{m_5;k_1,k_2;m_6}^{(\e)}
& =
-\sum_{s=0}^{m_5}(-2)^{-(m_6+m_5-s+1)}{m_6+m_5-s\choose m_6}\frac{1}{k_1!k_2!s!2^{k_1+k_2+s}}\\
&\qquad
\times\int_{|t|\ge \e^2} H_{2k_1}(x/\sqrt{2t})H_{2k_2}(x/\sqrt{2t})H_{2s}(x/\sqrt{2t})G(x,t)^3 \,dtdx.
\end{equs}


It is possible to explicitly carry out the integrals
of the type
\[
\int  
\Big(
\prod_{j=1}^{3} H_{n_{j}}(x/\sqrt{2t})
\Big)
G(t,x)^{3}\,dxdt
\]
Indeed, one has the following lemma.
However, having the values of these integrals
does not seem to lead us to more instructive formulae.
On the other hand,
with the following lemma we may 
compute the logarithmic constants on a computer, see Remark~\ref{rem:logs}.

\begin{lemma}
Define, for $n_{1},n_{2},n_{3} \in \Z_{\ge 0}$
\begin{equ}
A_{n_{1},n_{2},n_{3}}
\eqdef 
\int dx 
\Big(
\prod_{j=1}^{3} H_{n_{j}}(x/\sqrt{2t})
\Big)
G(t,x)^{3}
\end{equ}
Then $A_{n_{1},n_{2},n_{3}} = 0$ if $n_{1} + n_{2} + n_{3}$ is odd and otherwise one has
\begin{equs}
A_{n_{1},n_{2},n_{3}}
=
\frac{2^{-(n_{1}+n_{2}+n_{3})/2}}{3^{1/2}4\pi t }
\sum_{
\substack{
r_{1}, r_{2}, r_{3}\\
r_{1} + r_{2} + r_{3}
\le
(n_{1} + n_{2} + n_{3})/2}}&
\Big(
\prod_{j=1}^{3}
(-1)^{r_{j}}
3^{r_{j} - n_{j}/2}
\frac{n_{j}!}{r_{j}!(n_{j}-2r_{j})!}
\Big)\\
{}&
\qquad
\times
\frac
{
\Big(
\sum_{j=1}^{3} n_{j} - 2r_{j}
\Big)!}
{
\Big(
\sum_{j=1}^{3} n_{j}/2 - r_{j}
\Big)!}
\end{equs}	
\end{lemma}
\begin{proof}
This is a straightforward computation using a change of variable  $\sqrt{\frac{3}{2t}}x \mapsto y$ and an expansion for a triple product of scaled Hermite polynomials \cite{carlitz1962product} . 
\end{proof}

\begin{remark}\label{rem:logs}
With Mathematica we found that the logarithmic renormalization constants do sum to a non-zero constant for the second, third and fourth layer.
In fact, with help of Mathematica  we have
$C_{\e,1}^{(2)}=-\frac12 \frac{1}{4\sqrt{3}\pi} \log\e$
and $C_{\e,1}^{(3)}=\frac12 \frac{1}{4\sqrt{3}\pi} \log\e$ (consistent with \cite[Section~3.2]{NonGaussShen}), which cancel out;
we have
\[
C_{\e,2}^{(2)}=-\frac{85}{288} \frac{1}{4\sqrt{3}\pi} \log\e
\qquad
C_{\e,2}^{(3)}=\frac{47}{144} \frac{1}{4\sqrt{3}\pi} \log\e
\]
\[
C_{\e,3}^{(2)}=-\frac{995}{6912} \frac{1}{4\sqrt{3}\pi} \log\e
\qquad
C_{\e,3}^{(3)}=\frac{445}{3456} \frac{1}{4\sqrt{3}\pi} \log\e
\]
\[
C_{\e,4}^{(2)}=-\frac{5129851}{53747712} \frac{1}{4\sqrt{3}\pi} \log\e
\qquad
C_{\e,4}^{(3)}=\frac{1018585}{13436928} \frac{1}{4\sqrt{3}\pi} \log\e
\]
\end{remark}

\appendix
\section{A short analysis of $G_i$}
\label{A}
Let $\zeta\in\R$. We say that a kernel $G$ defined on $\R^d\setminus\{0\}$ or on a subset thereof
is of order $\zeta$ if for all multiindizes $k$ there exists a constant $C>0$ such that
\begin{equation}
\sup_{\|x\|_\s\leq 1} |D^k K(x)|\leq C\|x\|_\s^{\zeta-|k|_\s}.
\end{equation}
\begin{proposition}
\label{Prop:G}
The kernel $G_i$ is of order $-1$ for every $i\geq 0$.
\end{proposition}
\begin{proof}
The result is well known for the usual heat kernel, see for instance~\cite[Lemma 7.4]{Regularity}. If $i=1$, then note that after a partial integration we have the identity
\begin{equation}
G_1= \partial_x G *\partial_x G,
\end{equation}
and both kernels on the right hand side are of order $-2$. The claim therefore follows from~\cite[Lemma 10.14]{Regularity}. For $i>1$ we can write $G_i= G_{i-1}* \partial_x^2 G$ so that we can conclude as in the case $i=1$.
\end{proof}
%
%

\section{Useful identity}
\label{sec:identity}

The identity
\begin{equ}[eq:heatkernelidentity]
2G'*\wtd{G'} = G+\wtd{G}
\end{equ}
was used in \cite[Section 6]{KPZJeremy}
where the tilde is the reflection, i.e. 
$\wtd F(z):=F(-z)$. 
Iteratively applying this identity we can get useful identities 
for convolutions of (derivatives of) $G_m$. 
For instance we have
$4G_0' * \wtd{G_1'} = 2 \wtd{G_1} -G_0 -\wtd G_0$,
since
\[
4 G'* \wtd{G'} *\wtd{G''}  =2(G + \wtd{G} )*\wtd{G''} 
= 2 \wtd{G} *\wtd{G''}  -2G'*\wtd{G'} 
 = 2 \wtd{G}*\wtd{G''} -G-\wtd G
\]
where we used integration by parts to shift a derivative
$2 G * \wtd{G''} 
=  -2G'*\wtd{G'} $ in the second step.
When the indices are large this calculation 
can get more involved, for instance, we have
$4G_1' * \wtd{G_1'} = - G_1 - \wtd{G_1} + G_0 + \wtd G_0$, 
because 
\begin{equs}
4 &G''*G'* \wtd{G'} *\wtd{G''} 
=2 G'' *  (G + \wtd{G} )*\wtd{G''} 
=2 G'' *  G *\wtd{G''} +2 G'' *   \wtd{G} *\wtd{G''} \\
&=-2 G'' *  G' *\wtd{G'} -2 G' *   \wtd{G'} *\wtd{G''} 
= -G''*(G+\wtd G) - (G+\wtd G)*\wtd{G''} \\
&= -G*G'' - \wtd{G}*\wtd{G''} - G''*\wtd G - G*\wtd{G''}
= -G*G'' - \wtd{G}*\wtd{G''} + 2G'*\wtd {G'} \\
&= -G*G'' - \wtd{G}*\wtd{G''} +G+\wtd G \;.
\end{equs}

%
%
%
%
%
%

In order to obtain a 
general set of identities 
we define kernels $D_{i,j}$, for $(i, j) \in \Lambda \eqdef \{0,1,2,\dots\}^{2} $ via 
\begin{equ}[e:def-Dij]
D_{i,j}
\eqdef
G'_{i}
\ast
\wtd{G'_{j}}.
\end{equ}
For $i,j > 0$, one has the recursion relation
\[
D_{i,j}
=
-
\frac{1}{2}
\Big(
D_{i-1,j}
+
D_{i,j-1} \Big)\;.
\]
Indeed, making use of~\eqref{eq:heatkernelidentity}, we can write
\begin{equation}
\begin{aligned}
D_{i,j}&= (\partial_x^2 G)^{*i}*G'*\wtd{G'}*(\partial_x^2\wtd G)^{*j}
\\
&= \frac12
\Big((\partial_x^2 G)^{*i}*G*(\partial_x^2\wtd G)^{*j}
+(\partial_x^2 G)^{*i}*\wtd G*(\partial_x^2\wtd G)^{*j} \Big)
\\
&= -\frac12
\Big((\partial_x^2 G)^{*i}*G* \partial_x^2\wtd G*(\partial_x^2\wtd G)^{*(j-1)}
+(\partial_x^2 G)^{*(i-1)}*\partial_x^2 G*\wtd G*(\partial_x^2\wtd G)^{*j} \Big)
\end{aligned}
\end{equation}
and by shifting a derivative the above recursion relation follows.

When $i = 0, j > 0$ one has the recursion

\[
D_{0,j}
=
-\frac{1}{2} D_{0,j-1}
+\frac{1}{2}
G_{j}.
\]
Since $\wtd{D_{i,j}} = D_{j,i}$ we get, when $j=0$ and $i > 0$
\[
D_{i,0}
=
-\frac{1}{2} D_{i-1,0}
+\frac{1}{2}
\widetilde{G_{i}}.
\]
Finally, $D_{0,0} = \frac{1}{2}[G_{0} + \widetilde{G_{0}}]$. 

We now use these recursions to find formula for $D_{i,j}$.
A lattice path is a sequence of nearest neighbor edges (steps) of $\Lambda$ which satisfy the natural adjacency relation.  

Let $B\eqdef \{(x,y) \in \Lambda, i=0 \mbox{ or } j=0\}$ be the boundary of the discrete first quadrant. 
We denote by $W(i,j)$ the set of all lattice paths $\gamma$ which
\begin{itemize}
\item Start at $(i,j)$
\item Only move down or to the left
\item Terminate at a site of $B$ (note, paths are allowed to travel along $B$ for some time, but they always end at a site of $B$, they can't go negative.)
\end{itemize}
For a lattice path $\gamma$ we denote by $l(\gamma)$ the number of steps in $\gamma$. One then has the following formula

\[
D_{i,j}
=
\sum_{\gamma \in W(i,j)}
(-2)^{-l(\gamma)}
F(\gamma_{\mathrm{end}})
\]
where $\gamma_{\mathrm{end}}$ is the final site visited by $\gamma$ and $F$ is a map from the sites of $B$ to kernels given as follows: $F(0,0) = \frac{1}{2}(G + \widetilde{G})$ and
\[
F(0,j)
=
\frac{1}{2} G_{j}
\qquad
\mbox{and}
\qquad
F(i,0)
=
\frac{1}{2} \widetilde{G_{i}}.
\]

For fixed $(i,j) \in \Lambda$ let $B(i,j)$ be the set of sites of $B$ one can reach via walks in $W(i,j)$. 
We then get
\[
D_{i,j}
=
\sum_{(x,y) \in B(i,j)}
(-2)^{-(i+j - x - y)}
\binom{i + j - x - y}{i-x}
F(x,y),
\]
where $\binom{i + j - x - y}{i-x}$ counts the number of paths
from $(i,j)$ to $(x,y)$.
Equivalently,
\begin{equ}[e:Dij-formula]
D_{i,j}
=
- \sum_{k =0}^i
(-2)^{-(i+j - k+1)}
\binom{i + j - k}{j} \wtd{G_k}
-\sum_{k =0}^j
(-2)^{-(i+j - k +1)}
\binom{i + j - k}{i} G_k \;.
\end{equ}

\endappendix

\bibliographystyle{Martin}

\bibliography{refs}

\end{document}